\documentclass[11pt]{amsart}

\usepackage{latexsym, amssymb, amsthm,color}
\usepackage[centertags]{amsmath}
\usepackage{pictexwd}
\usepackage[normalem]{ulem}
\usepackage{hyperref}

\definecolor{GREEN}{rgb}{0,1,0}
\definecolor{green4}{rgb}{.1,.5,.1}
\definecolor{blue}{rgb}{0,0,1}
\definecolor{gray}{rgb}{0.5,0.5,0.5}

\newcommand{\mathrmtiny}[1]{\mathrm{\scriptscriptstyle #1}}
 \newcommand{\ep}{\end{proof}}

 \newif\ifpctex

\newtheorem{theorem}{Theorem}[section]
\newtheorem{definition}{Definition}[section]

\newtheorem{cond}[definition]{Condition}
\newtheorem{proposition}[definition]{Proposition}
\newtheorem{lemma}[definition]{Lemma}

\newtheorem{corollary}[definition]{Corollary}
\newcommand{\beCond}[2]{\Rand{\vspace{0,6cm}\tt #1}\begin{cond}[#2]
\label{#1}} \theoremstyle{definition}
\newtheorem{remark}[definition]{Remark}
\newtheorem{example}[definition]{Example}
\numberwithin{equation}{section}
\newtheoremstyle{step}{3pt}{0pt}{\itshape}{}{\bf}{}{.5em}{}

\theoremstyle{step} \newtheorem{step}{Step}

\newcommand{\Rand}[1]{\marginpar{#1}}
\newcommand{\be}[1]{\begin{equation}\label{#1}}
\newcommand{\ee}{\end{equation}}
\newcommand{\bew}[1]{\Rand{\vspace{0,6cm}\tt #1}\begin{equation*}\label{#1}}
\newcommand{\eew}{\end{equation*}}
\newcommand{\bea}[1]{\Rand{\vspace{0,6cm}\tt #1}\begin{eqnarray*}\label{#1}}
\newcommand{\eea}[1]{\end{eqnarray*}}

\newcommand{\beL}[2]{\Rand{\vspace{0,6cm}\tt #1}\begin{lemma}[#2]\label{#1}}
\newcommand{\beD}[2]{\Rand{\vspace{0,6cm}\tt #1}\begin{definition}[#2]\label{#1}}
\newcommand{\beT}[2]{\Rand{\vspace{0,6cm}\tt #1}\begin{theorem}[#2]\label{#1}}
\newcommand{\beP}[2]{\Rand{\vspace{0,6cm}\tt #1}\begin{proposition}[#2]\label{#1}}
\newcommand{\beC}[1]{\Rand{\vspace{0,6cm}\tt #1}\begin{corollary}\label{#1}}
\newcommand{\beR}[1]{\Rand{\vspace{0,6cm}\tt #1}\begin{remark}[#1]\label{#1}}

\newcommand{\Tto}{{_{\displaystyle\Longrightarrow\atop t\to\infty}}}

\newcommand{\tno}{{_{\displaystyle\longrightarrow\atop n\to\infty}}}
\newcommand{\tko}{{_{\displaystyle\longrightarrow\atop k\to\infty}}}
\newcommand{\TNo}{{_{\displaystyle\Longrightarrow\atop N\to\infty}}}
\newcommand{\tNo}{{_{\displaystyle\longrightarrow\atop N\to\infty}}}

\DeclareMathAlphabet{\mathpzc}{OT1}{pzc}{m}{it}

\begin{document}

\title[The grapheme-valued Wright-Fisher diffusion with mutation]
{{\large The grapheme-valued Wright-Fisher diffusion\\ with mutation}}

\author{Andreas Greven}
\address{Andreas Greven\\
Department Mathematik, Universit\"at Erlangen-N\"urnberg, Cauerstrasse 11, 91058 Erlangen, Germany}
\email{greven@mi.uni-erlangen.de}

\author{Frank den Hollander}
\address{Frank den Hollander\\
Mathematisch Instituut, Universiteit Leiden, Niels Bohrweg 1, 2333 CA  Leiden, The Netherlands}
\email{denholla@math.leidenuniv.nl}

\author{Anton Klimovsky}
\address{Anton Klimovsky\\
Institut f\"ur Mathematik, Emil-Fischer-Strasse 30, 97074 W\"urzburg, Germany}
\email{anton.klymovskiy@mathematik.uni-wuerzburg.de}

\author{Anita Winter}
\address{Anita Winter \\
Fakult\"at f\"ur Mathematik, Universit\"at Duisburg-Essen, Campus Essen, Universit\"atsstra{\ss}e 2, 45132 Essen, Germany}
\email{anita.winter@uni-due.de}

\thanks{AG was supported by the Deutsche Forschungsgemeinschaft (through grant DFG-GR 876/16-2 of SPP-1590). FdH was supported by the Netherlands Organisation for Scientic Research (through NWO Gravitation Grant NETWORKS-024.002.003), and by the Alexander von Humboldt Foundation. AK was supported by the Deutsche Forschungsgemeinschaft (through Project-ID 443891315 within SPP 2265 and through Project-ID 412848929). AW was supported by the Deutsche Forschungsgemeinschaft (through Project-ID 444091549 within SPP-2265).}

\thispagestyle{empty}
\date{\today}



\begin{abstract}
In \cite{AthreyadenHollanderRoellin2021} models from population genetics were used to define stochastic dynamics in the space of graphons arising as continuum limits of dense graphs. In the present paper we exhibit an example of a simple neutral population genetics model for which this dynamics is a Markovian diffusion that can be characterised as the solution of a martingale problem. In particular, we consider a Markov chain in the space of finite graphs that resembles a Moran model with resampling and mutation. We encode the finite graphs as graphemes, which can be represented as a triple consisting of a vertex set, an adjacency matrix and a sampling measure. We equip the space of graphons with convergence of sample subgraph densities and show that the grapheme-valued Markov chain converges to a grapheme-valued diffusion as the number of vertices goes to infinity. We show that the grapheme-valued diffusion has a stationary distribution that is linked to the Poisson-Dirichlet distribution. In a companion paper \cite{GrevendenHollanderKlimovskyWinter2023} we build up a general theory for obtaining grapheme-valued diffusions via genealogies of models in population genetics.
\end{abstract}

\maketitle

{\tiny \tableofcontents}

\newpage


\section{Introduction and background}
\label{S:introduction}

In \cite{AthreyadenHollanderRoellin2021}, stochastic processes with values in the space of \emph{graphons} where constructed based on dynamics from population genetics. Even though this paper provided a large class of examples, intended as a \emph{proof of concept}, no general theory was developed and consequently many structural questions remained open. For instance, it was not clear whether or not the examples that were constructed are strong Markov processes that can be described by a generator acting on a dense class of test functions on the graphon space.
Moreover, all the examples had trivial equilibria concentrated on constant graphons.

The goal of the present paper is to give an example, based on \emph{resampling} and \emph{mutation}, to construct a \emph{graphon-valued diffusion} via a \emph{well-posed martingale problem}. To do so, we equip the space of graphons with convergence of subgraph densities, which can be viewed as a \emph{sample adjacency distribution}. This follows approaches of sample convergence of distance matrices (\cite{GrevenPfaffelhuberWinter2009}, \cite{LoehrVoisinWinter2015}) and shapes (\cite{LoehrMytnikWinter2020}, \cite{LoehrWinter2021}) used in the context of random trees.

We start with finite Markov chains (graph-valued, with values in the space of adjacency matrices) and provide a \emph{duality relation} for a finite-graph-valued Markovian dynamics that can be lifted to a limiting graphon-valued Markovian dynamics. We show that the limiting dynamics takes place in the subspace of \emph{graphemes}, which are countable graphs embedded in a Polish space that can be viewed as graphs with an adjacency structure. We use a criterion to decide whether or not a graphon is a grapheme in terms of the \emph{entropy of the sample adjacency distribution}, and apply this criterion to come up with a criterion for relative compactness. Since the finite-graph-valued dynamics moves through the space of finite graphs with \emph{complete components}, we introduce the subspace of complete component graphons and show that it is a closed subspace of the space of graphons. The limiting diffusion is introduced by means of a martingale problem. We consider so-called sample adjacency polynomials, which describe the subtree count densities, and use these as test functions to write down a well-posed martingale problem. We verify that the limit of the stationary distribution of the finite-graph-valued dynamics is the stationary distribution of the graphon-valued dynamics.

In a companion paper \cite{GrevendenHollanderKlimovskyWinter2023}, we build up a general theory for obtaining grapheme-valued diffusions via \emph{genealogies} of models in population genetics. Like in \cite{AthreyadenHollanderRoellin2021}, we exploit the machinery developed for models from population dynamics, but we add the idea to also exploit the tree structure behind the genealogy of populations, which is both natural and convenient. The specific model analysed in the present paper serves as an illustration.

The rest of the paper is structured as follows. In Section~\ref{S:finite}, we introduce the poaching model with self-employment, which is our finite-graph-valued Markov chain with a fixed vertex set. To be able to evaluate this dynamics with test functions inducing our notion of convergence, we translate it to a Markov chain taking values in the adjacency matrices. In order to establish the stationary distribution, we further relate our model to a Moran model with mutation. In Section~\ref{S:graphemes}, we recall the notion of a graphon, provide examples, and discuss several topological issues that we will further use. In particular, we introduce the subspace of graphemes and graphs with complete graph components. In Section~\ref{S:WFgrapheme}, we formulate the martingale problem of the diffusion limit as the number of vertices goes to infinity and study its properties.


\section{Finite-state Markov chains}
\label{S:finite}

In this section, we consider three closely related Markov chains. In Subsection~\ref{Sub:Ngraph} we consider a Markov chain on the space of finite graphs with a fixed number of vertices. In Subsection~\ref{Sub:adjacency}, we consider the functional of this Markov chain that observes the dynamics of the adjacency matrix. In Subsection~\ref{Sub:infiniteAllele}, we link the two Markov chains to the Moran model with mutation.


\subsection{The Markov chain with values in the space of finite graphs}
\label{Sub:Ngraph}

By an $N$-graph, we mean an undirected labelled graph $G=(V,E)$ with $\# V=N$ vertices, labelled by the index set $[N] := \{1, 2, \ldots, N\}$. We write
\begin{equation}
\label{f:178}
\mathbb{G}^N := \text{ the space of $N$-graphs.}
\end{equation}
We introduce the $N$-graph {\em poaching model with self-employment}, i.e., the following continuous-time Markov chain $X=(X_t)_{t \ge 0}$ on the space $\mathbb{G}^N$ of $N$-graphs:
\begin{itemize}
\item
{\bf Poaching.} For each pair $(v,v')\in V^2$, at rate $1$, the vertex $v$ forces the vertex $v'$ to give up all its connections and instead reconnect with the vertex $v$ as well as with all its adjacent vertices, i.e., with all $v''\in V$ such that $\{v,v''\}\in E$.
\item
{\bf Self-employment.}
For each vertex $v\in V$, at rate $\mu\ge 0$, the vertex $v$ isolates itself by disconnecting from all other vertices.
\end{itemize}
For this Markov chain, the infinitesimal generator $\Omega^N_{\mathrmtiny{ graph}}$ acts on a function $f\colon\,\mathbb{G}^N\to\mathbb{R}$ as follows: for all $G=(V,E)$,
\begin{equation}
\label{f:001}
(\Omega^N_{\mathrmtiny{ graph}}f)(G) = \sum_{v_1\not =v_2\in V} \big[f(G^{(v_1,v_2)})-f(G)\big]
+ \mu\sum_{v\in V} \big[f(G^{v})-f(G)\big],
\end{equation}
where $G^{(v_1,v_2)}$ denotes the $N$-graph $G$ after vertex $v_2$ was poached by vertex $v_1$, and $G^{v}$ denotes the $N$-graph $G$ after the vertex $v$ became self-employed.

Denote by $\mathbb{G}^N_{\mathrmtiny{ complete}}$ the subspace of $\mathbb{G}^N$ consisting of graphs whose connected components are all complete, i.e.,
\begin{equation}
\label{f:179}
\begin{aligned}
\mathbb{G}^N_{\mathrmtiny{ complete}}
:= \big\{G=(V,E)\in\mathbb{G}^N\colon\,
&\{v_2,v_3\}\in E\text{ if } \{v_1,v_2\},\{v_1,v_3\} \in E,\\
&\{v_2,v_3\}\not\in E\text{ if } \{v_1,v_2\},\{v_1,v_3\} \not\in E\big\}.
\end{aligned}
\end{equation}
Note that any $G\not\in \mathbb{G}^N_{\mathrmtiny{ complete}}$ is a transient state of $X$. Moreover, on $\mathbb{G}^N_{\mathrmtiny{ complete}}$ the Markov chain is irreducible. Therefore, the Markov chain converges, for all initial $N$-graphs, to the unique stationary distribution $\pi^N_{\mathrmtiny{ graph}}$ that is supported on $\mathbb{G}^N_{\mathrmtiny{ complete}}$.

To describe $\pi^N_{\mathrmtiny{ graph}}$, for each $G=(V,E)\in\mathbb{G}^N$, let two vertices fall into the same connected component if and only if their labels belong to the same partition element. Define
\begin{equation}
\label{f:180}
\mathrm{Comp}(G) := \text{ set of connected components of $G$,}
\end{equation}
and consider the map $g\colon\,\mathbb{G}^N\to{\mathcal N}_N(\mathbb{N})$ given by
\begin{equation}
\label{f:022}
g(G) := \sum_{\varpi\in\mathrm{Comp}(G)} \delta_{\#\varpi},
\end{equation}
where ${\mathcal N}^N(\mathbb{N})$ is the set of integer-valued measures $\nu$ on $\mathbb{N}$ with $\sum_{k\in\mathbb{N}}k\nu(k)=N$.

\begin{proposition}[Stationary distribution and ergodicity]
For $N\in\mathbb{N}$ and $G\in\mathbb{G}^N$, define
\begin{equation}
\label{f:189}
\pi^N_{\mathrmtiny{ graph}}(G) := \prod_{j=1}^N \frac{\mu^{g(G)(j)}}{(\mu+(j-1))}
\,\mathbf{1}_{\mathbb{G}^N_{\mathrmtiny{ graph}}}(G).
\end{equation}
Then, for all initial states $G\in\mathbb{G}^N$,
\begin{equation}
\label{f:010e}
{\mathcal L}_G(X^N_t) \Tto \pi^N_{\mathrmtiny{ graph}}.
\end{equation}
\label{P:012}
\end{proposition}

Later on, we declare two $N$-graphs $G=(V,E)$ and $G'=(V',E')$ to be equivalent if and only if they are isomorphic to each other, i.e., if there is a one-to-one map $\gamma\colon\,V \to V'$ such that $\{v,w\}\in E$ if and only if  $\{\gamma(v),\gamma(W)\} \in E'$. If $G=(V,E)$ and $G'=(V',E')$ are equivalent, we write $G\sim G'$, and let
\begin{equation}
\label{f:187}
\big[G=(V,E)\big] := \big\{G'=(V',E') \text{ with } G\sim G'\big\}
\end{equation}
be the isomorphism class of $G$, and
\begin{equation}
\label{f:190}
\mathbb{G}^{N,\simeq} := \big\{[G];\,G\in\mathbb{G}^N\big\}.
\end{equation}
Obviously, if $G\in\mathbb{G}^N_{\mathrmtiny{ complete}}$, then $G'\in[G]$ if and only of $g(G)=g(G')$. Moreover, if $G\in\mathbb{G}^N_{\mathrmtiny{ complete}}$, then
\begin{equation}
\label{f:188}
\#\big[G\big] = \prod_{j=1}^N \Big(\big(j!\big)^{g(G)(j)-1}\big(g(G)(j)\big)!\Big)^{-1}.
\end{equation}

\begin{definition}[Multivariate Ewens Distribution $N$-graph]
{\rm Fix $N\in\mathbb{N}$. We refer to a random element $G^N_{\mathrmtiny{ MED}}$ in $\mathbb{G}^N_{\mathrmtiny{ complete}}$ as the \textit{multivariate Ewens distribution (MED) $N$-graph} if
\begin{equation}
\label{f:177}
\begin{aligned}
\mathbb{P}\big(g(G^N_{\mathrmtiny{ MED}})=\nu\big)=\pi^N_{\mathrmtiny{ MED}}\big(\nu\big),
\end{aligned}
\end{equation}
where $\pi^N_{\mathrmtiny{ MED}}$ denotes the multivariate Ewens distribution on ${\mathcal N}^N(\mathbb{N})$ given by
\begin{equation}
\label{f:010a}
\pi^N_{\mathrmtiny{ MED}}\big(\nu\big) = \frac{N!}{\mu(\mu+1) \cdots (\mu+N-1)} \prod_{j=1}^N \frac{\mu^{\nu(j)}}{j^{\nu(j)}(\nu(j))!}.
\end{equation}
\label{Def:001}
} \hfill$\spadesuit$
\end{definition}

The MED was introduced in \cite{Ewens1972} (compare also with \cite{DonnellyTavari1986}), and gives the probability of the partition of a sample of $N$ selectively equivalent genes into a number of different gene types (alleles), either exactly in some models of genetic evolution or as a limiting distribution (as the population size tends to infinity) in others. It was discovered independently in \cite{Antoniak1974} in the context of Bayesian statistics. See \cite[Chapter~41]{JohnsonKotzBalakrishnan1997} for a survey on the MED.

We claim that MED-$N$-graphs are the unique stationary random graphs for the poaching model with self-employment lifted to $\mathbb{G}^{N,\simeq}_{\mathrmtiny{}}$.

\begin{proposition}[Ergodicity]
Let $G^N_{\mathrmtiny{ MED}}$ be the MED $N$-graph, and $X^N=(X^N_t)_{t \ge 0}$ be the $\mathbb{G}^N$-valued Markov chain with poaching and self-employment. Then, for all initial states $G\in\mathbb{G}^N$,
\begin{equation}
\label{f:010e*}
{\mathcal L}_G\big(\big[X^N_t\big]\big)\Tto{\mathcal L}\big(\big[G^N_{\mathrmtiny{ MED}}\big]\big).
\end{equation}
\label{P:001}
\end{proposition}

From Proposition~\ref{P:001}, we obtain  the following (see, for example, \cite{JohnsonKotzBalakrishnan1997}).

\begin{corollary}[Properties of the equilibrium graph]
Let $\{\xi_i,\,i\in\mathbb{N}\}$ be a family of independent Bernoulli random variables with $\mathbb{P}(\xi_i=1)=1-\mathbb{P}(\xi_i=0)=(1+\frac{i-1}{\mu})^{-1}$. Then,
\begin{equation}
\label{f:010c}
{\mathcal L}\big(\#\mathrm{Comp}(G^N_{\mathrmtiny{ MED}})\big) = {\mathcal L}\Big(\sum_{i=1}^N \xi_i\Big).
\end{equation}
\label{Cor:001}
\end{corollary}

The appearance of the MED is not by accident: the graph dynamics with poaching and self-employment is closely related to one of the simplest neutral stationary population models. We will introduce this model in Subsection~\ref{Sub:infiniteAllele}, where we also provide the proof of Proposition~\ref{P:001}.


\subsection{The Markov chain with values in the space of adjacency matrices}
\label{Sub:adjacency}

In this subsection, we introduce a second Markov chain, taking values in the space of adjacency matrices, that in some sense also describes the graph dynamics with poaching and self-employment.

\begin{definition}[Adjacency matrix]
{\rm For a finite subset $\Xi\subset\mathbb{N}$, we call a symmetric $\Xi\times \Xi$-matrix  $A=(A_{i,j})_{i,j\in \Xi}$ an \textit{adjacency matrix} on $\Xi$ if $A_{i,j}\in\{0,1\}$ and $A_{i,i}=0$ for all $i,j\in \Xi$. We let
\begin{equation}
\label{e:027}
\mathfrak{A}_\Xi := \text{ the space of all adjacency matrices on $\Xi$,}
\end{equation}
and
\begin{equation}
\label{e:027*}
\mathfrak{A}_N := \bigcup_{\Xi\subseteq[N]} \mathfrak{A}_\Xi
\end{equation}
the space of all adjacency matrices on a finite subset of $[N]$.} \hfill$\spadesuit$
\label{Def:005}
\end{definition}

For $N\in\mathbb{N}$ and $\Xi \subset [N]$, we are interested in the $\mathfrak{A}_\Xi$-valued Markov chain $M^\Xi=(M^\Xi_t)_{t \ge 0}$ that, given the current state $A\in\mathfrak{A}_\Xi$, has the following dynamics:
\begin{itemize}
\item
{\bf Duplication.}
For all $i\not=j\in \Xi$, at rate $1$ $A$ is replaced by $\sigma_{i,j}(A) \in \mathfrak{A}_\Xi$, where
\begin{equation}
\label{f:002}
\big(\sigma_{i,j}(A)\big)_{k,l} := 
\begin{cases}
A_{k,l}, & \text{if }j\not\in\{k,l\},\\
A_{i,l}+\delta_{i,l}, &\text{if } j=k, \\
A_{k,i}+\delta_{k,i}, &\text{if } j=l.
\end{cases}
\end{equation}
\item
{\bf Grounding.}
For all $i\in \Xi$, at rate $\mu\ge 0$ $A$ is replaced by $0_{i}(A) \in \mathfrak{A}_\Xi$, where
\begin{equation}
\label{f:003}
\big(0_{i}(A)\big)_{k,l} := 
\begin{cases}
A_{k,l}, &\text{if } i\not\in\{k,l\}, \\
0, &\text{if } i\in\{k,l\}.
\end{cases}
\end{equation}
\end{itemize}
In what follows, we will refer to this Markov chain as the {\em duplication model with grounding}. The infinitesimal generator $\Omega^\Xi_{\mathrmtiny{ adjacency}}$ of this Markov chain is the operator acting on a function $\phi\colon\,\mathfrak{A}_\Xi \to \mathbb{R}$ as follows: for all $A\in\mathfrak{A}_\Xi$,
\begin{equation}
\label{f:004}
(\Omega^\Xi_{\mathrmtiny{ adjacancy}}\phi)(A) = \sum_{i\not =j\in\Xi} \big[\phi(\sigma_{i,j}(A))-\phi(A)\big]
+ \mu \sum_{i\in\Xi} \big[\phi(0_i(A))-\phi(A)\big].
\end{equation}

For an $N$-graph $G=(V,E) \in \mathbb{G}^N$, a subset $\Xi \subseteq [N]$, and a vector $(x_1,\ldots,x_{\#\Xi}) \in V^{\Xi}$ with distinct entries, write
\begin{equation}
\label{f:124}
\mathrm{adj}\big((x_1,\ldots,x_{\#\Xi}) := \big(\mathbf{1}_E(\{x_i,x_j\}\big)_{i,j\in\Xi}
\end{equation}
for the {\em adjacency matrix on $\Xi$ generated by the vector $(x_1,\ldots,x_{\Xi})\in V^{\Xi}$}. Note that the $\mathfrak{A}_{[N]}$-valued duplication model with grounding is linked with the $\mathbb{G}^N$-valued poaching model with self-employment as follows: if $f\colon\,\mathbb{G}^N \to \mathbb{R}$ has the form $f(G) = \phi\circ\mathrm{adj}((v_{(1)},\ldots,v_{(N)}))$ for $G=(V,E)$ and some ordering $(v_{(1)},\ldots,v_{(N)})$ of $V$, then
\begin{equation}
\label{f:182}
{\mathcal L}_{G}\Big(\big(f(X^N_t)\big)_{t \ge 0}\Big) ={\mathcal L}_{\mathrm{adj}((v_{(1)},\ldots,v_{(N)}))}
\Big(\big(\phi(M^{[N]}_t)\big)_{t\ge 0}\Big).
\end{equation}

The main goal of the present paper is to show that the poaching model with self-employment converges to a diffusion with values in the space of continuum limits of finite graphs. The corresponding notion of convergence will be introduced in Section~\ref{S:graphemes} and relies on the convergence of densities of finite subgraphs. We are particularly interested in the dynamics of these subgraph densities. Subgraph densities are given by the probability that the subgraph spanned by a sample of vertices agrees with a given graph of the prescribed sample size. They are captured by test functions of the form
\begin{equation}
\label{f:122}
\Phi^{\Xi,A}(G) := \frac{(N-k)!}{N!} \sum_{(x_1,\ldots,x_{\#\Xi}) \in V^\Xi,\#\{x_1,x_2,\ldots,x_k\} = \#\Xi}\,
\mathbf{1}_A\Big(\mathrm{adj}\big((x_1,\ldots,x_{\#\Xi})\big)\Big),
\end{equation}
where $\Xi\subseteq[N]$ and $A\in\mathfrak{A}_{\Xi}$. We refer to $\Phi^{[k],A}(G)$ as the density of the $k$-graph with  $A\in\mathfrak{A}_{[k]}$ within the $N$-graph $G=(V,E)$.

A simple consequence of (\ref{f:182}) is the following.

\begin{lemma}[Generator action on subgraph densities]
Let $G=(V,E)$, $N\in\mathbb{N}$, $\Xi\subseteq[N]$ and $A\in\mathfrak{A}_{\Xi}$. Then,
\begin{equation}
\label{f:125b}
\begin{aligned}
(\Omega^N_{\mathrmtiny{ graph}}\Phi^{\Xi,A})(G)
&= \frac{(N-k)!}{N!} \sum_{\scriptscriptstyle \underline{x}\in V^{\Xi}, \atop
 \#\{x_1,\ldots,x_{\#\Xi}\}=\#\Xi } \Omega^{\Xi}_{\mathrmtiny{ adjacency}}\,
\mathbf{1}_A\big(\mathrm{adj}\big(\underline{x}\big)\big).
\end{aligned}
\end{equation}
\label{L:002}
\end{lemma}

Lemma~\ref{L:002} allows us to express the action of the generator of the $N$-graph-valued poaching with self-employment model on subgraph densities as follows.

\begin{proposition}[Graph dynamics acting on subgraph densities]
Let $N\in\mathbb{N}$,  $\Xi\subseteq[N]$ and $A\in\mathfrak{A}_\Xi$. Then, for $G=(V,E)$,
\begin{equation}
\label{f:123}
\begin{aligned}
&(\Omega^N_{\mathrmtiny{ graph}}\Phi^{\Xi,A})(G)\\
&= \sum_{i\not=j\in\Xi} \Big[\mathbf{1}_A\big\{A_{j,l} = A_{i,l}+\delta_{i,l},\,\,\forall\,\ell\in\Xi\setminus\{j\}\big\}\,\Phi^{\Xi\setminus\{j\},
\vartheta_{j}(A)}(G)-\Phi^{\Xi,A}(G)\Big]\\
&\quad + \mu \sum_{j\in\Xi} \Big[\mathbf{1}_A\big\{A_{j,l}=0,\,\,\forall\,\ell\in\Xi\big\}\,\Phi^{\Xi\setminus\{j\},\vartheta_j(A)}(G)-\Phi^{\Xi,A}(G)\Big],
\end{aligned}
\end{equation}
where $\vartheta_j(A)=(A_{k,l})_{k,l\in\Xi\setminus\{j\}}$ is the adjacency matrix $A$ on $\Xi$ with the entries in the $j^{\mathrm{th}}$ row and column erased.
\label{P:002}
\end{proposition}

\begin{proof}
Let $G=(V,E)$, $N\in\mathbb{N}$, $\Xi\subseteq[N]$, $\underline{x}=(x_1,\ldots,x_{\#\Xi})\in V^{\Xi}$ with distinct entries, and $A\in\mathfrak{A}_{\Xi}$. Then,
\begin{equation}
\label{f:125}
\begin{aligned}
&(\Omega^N_{\mathrmtiny{ adjacency}}\mathbf{1}_A)\big(\mathrm{adj}\big(\underline{x}\big)\big)\\
&= \sum_{i\not =j\in\Xi} \Big[\mathbf{1}_A\big(\sigma_{i,j}\big(\mathrm{adj}(\underline{x})\big)\big)
-\mathbf{1}_A\big(\mathrm{adj}(\underline{x})\big)\Big]
+ \mu \sum_{i\in\Xi} \Big[\mathbf{1}_A\big(0_{i}\big(\mathrm{adj}(\underline{x})\big)\big)
-\mathbf{1}_A\big(\mathrm{adj}(\underline{x})\big)\Big]\\
&= \sum_{i\not =j\in\Xi}\Big[\mathbf{1}\{A_{l,j}=A_{l,i}+\delta_{l,i}\,\,\forall\, l\in\Xi\setminus\{j\}\}
\,\mathbf{1}_{\vartheta_j(A)}\big(\vartheta_{j}\big(\mathrm{adj}(\underline{x})\big)\big)
-\mathbf{1}_A\big(\mathrm{adj}(\underline{x})\big)\Big]\\
&\quad + \mu \sum_{i\in\Xi}\Big[\mathbf{1}\{A_{l,i}=0,\,\,\forall\, l\in\Xi\}
\mathbf{1}_{\vartheta_i(A)}\big(\vartheta_i\big(\mathrm{adj}(\underline{x})\big)\big)
- \mathbf{1}_A\big(\mathrm{adj}(\underline{x})\big)\Big].
\end{aligned}
\end{equation}
Here we use that $\sigma_{i,j}(B)=A$ if and only if $\vartheta_j(B)=\vartheta_j(A)$ and $A_{l,j}=A_{l,i}+\delta_{l,i}$ for all $l\in\Xi\setminus\{j\}$, and that $0_i(B)=A$ if and only if $\vartheta_i(B)=\vartheta_i(A)$ and $A_{k,i}=0$ for all $i\in\Xi$. Averaging the latter relation over all $\#\Xi$-element subsets of $V$, we get the claim.
\end{proof}

We close this subsection with a duality relation which, due to the special algebraic form of the generator as stated in Lemma~\ref{L:002}, can later be lifted to a duality relation for the limiting diffusion we want to construct. To consider the time-reversed Markov chain, we fix $N\in\mathbb{N}$, denote by $q_{\tilde{A},A}$ the rate at which the $\mathfrak{A}_{[N]}$-valued duplication model with grounding jumps from $\tilde{A}$ to $A$, and put
\begin{equation}
\label{f:011}
\overleftarrow{q}_{A,\tilde{A}}:=q_{\tilde{A},A}, \qquad \tilde{A},A\in\mathfrak{A}_{[N]}.
\end{equation}
We can again characterise the backward chain analytically through its generator $\overleftarrow{\Omega}^N_{\mathrmtiny{ adjacancy}}$ acting on functions $\phi\colon\,\mathfrak{A}_N\to\mathbb{R}$ as follows:
\begin{equation}
\label{f:127}
\overleftarrow{\Omega}^N_{\mathrmtiny{ adjacancy}}\phi\big(A\big)
= \sum_{\tilde{A}\in\mathfrak{A}_{[N]}}\overleftarrow{q}_{A,\tilde{A}} \big[\phi(\tilde{A})-\phi(A)\big].
\end{equation}
As observed and exploited in \cite{Seidel2015}, if the state space of a Markov chain is finite, then there is a Feynman-Kac duality relation to the time-reversed Markov chain, with the {\em duality function} given by
\begin{equation}
\label{f:012}
H(\tilde{A},A) := \mathbf{1}\{\tilde{A}=A\}, \qquad \tilde{A},A\in\mathfrak{A}_{[N]},
\end{equation}
and with the {\em potential} being the difference of the total rates in the backward and forward time Markov chain, i.e.,
\begin{equation}
\label{f:019}
\begin{aligned}
\beta^\mu_N(A)
&:= \sum_{\tilde{A}\in\mathfrak{A}_{[N]}} \Big[\overleftarrow{q}_{A,\tilde{A}}-q_{A,\tilde{A}}\Big]\\
&= \big[N(N-1)-\#\big\{1\le i\not =j\le N\colon\,A_{j,k}=A_{i,k}+\delta_{i,k}\,\,\forall\, k\in[N]\big\}\big]\\
&\quad + \mu \Big(N-\#\big\{i\in[N]:\,A_{k,i}=0\,\,\forall\, k\in[N]\big\}\Big).
 \end{aligned}
\end{equation}

\begin{proposition}[Feynman-Kac duality]
Let  $M=(M_t)_{t \ge 0}$ and $\overleftarrow{M}=(\overleftarrow{M}_t)_{t\ge 0}$ be the above $\mathfrak{A}_{[N]}$-valued Markov chains forward and backward in time, respectively. Then, for all $\tilde{A},A\in\mathfrak{A}_{[N]}$,
\begin{equation}
\label{f:020}
\mathbb{P}_{A}\big(M_t=\tilde{A}\big) = \mathbb{E}_{\tilde{A}}\Big[\mathbf{1}_{A}(\overleftarrow{M}_t)
\exp\Big(\int_0^t\beta^\mu_N(\overleftarrow{M}_s)\,\mathrm{d}s\Big)\Big].
\end{equation}
\label{P:006}
\end{proposition}

\begin{proof}
For all $\tilde{A},A\in\mathfrak{A}_{[N]}$,
\begin{equation}
\label{f:126}
\begin{aligned}
(\Omega^N_{\mathrmtiny{ adjacancy}}\mathbf{1}_{\tilde{A}})(A)
&= (\Omega^N_{\mathrmtiny{ adjacancy}}H)(\tilde{A},A)\\
&= \sum_{B\in\mathfrak{A}_{[N]}}q_{\tilde{A},B}\Big[H\big(B,A\big)-H\big(\tilde{A},A\big)\Big]\\
&= \sum_{B\in\mathfrak{A}_{[N]}}\overleftarrow{q}_{B,\tilde{A}}H(B,A)
-H\big(\tilde{A},A\big)\sum_{B\in\mathfrak{A}_{[N]}}q_{\tilde{A},B}\\
&= \overleftarrow{q}_{A,\tilde{A}}-H(\tilde{A},A)\left(\sum_{B\in\mathfrak{A}_{[N]}}\overleftarrow{q}_{\tilde{A},B}
-\beta^\mu_N(\tilde{A})\right)\\
&= \sum_{B\in\mathfrak{A}_{[N]}}\overleftarrow{q}_{A,B}\Big[H(\tilde{A},B)-H(\tilde{A},A)\Big]
+\beta^\mu_N(\tilde{A})H\big(\tilde{A},A\big)\\
&= (\overleftarrow{\Omega}^N_{\mathrmtiny{ adjacancy}}H)(\tilde{A},A)+\beta^\mu_N(\tilde{A})H(\tilde{A},A).
\end{aligned}
\end{equation}
Due to the finiteness of $V$, the potential $\beta^\mu_N(\tilde{A}))$ is bounded from above and so the claim follows (see e.g.\ \cite[Theorem~4.4.11]{EthierKurtz1986}).
\end{proof}


\subsection{The Moran model with mutation}
\label{Sub:infiniteAllele}

Let $\mathbb{K}:=[0,1]$ and $N\in\mathbb{N}$. An element $(x_1,\ldots,x_N)$ $\in \mathbb{K}^N$ is referred to as an $N$-population with genetic types in $\mathbb{K}=[0,1]$. In this subsection, we consider one of the simplest selectively neutral $\mathbb{K}^N$-valued Markov chains $Y=(Y_t)_{t \geq 0}$. Given the current $N$-population $(x_1,\ldots,x_N)$, the evolution of $Y$ is given by:
\begin{itemize}
\item
{\bf Resampling.}
For each pair $i,j \in [N]$ with $i < j$, at rate $1$, the $j^{\mathrm{th}}$ individual is pushed out of the population and is replaced by a clone of the $i^{\mathrm{th}}$ individual.
\item
{\bf Mutation.}
For each $i \in [N]$, at rate $\mu\ge 0$, the genetic type of the $i^{\mathrm{th}}$ individual is replaced by a new type uniformly sampled from $\mathbb{K}$.
\end{itemize}
The infinitesimal generator $\Omega^{N}_{\mathrmtiny{ WFmut}}$ of this Markov chain acts on bounded continuous functions $f\colon\,\mathbb{K}^N\to\mathbb{R}$ as follows: for all $\underline{x}\in \mathbb{K}^N$,
\begin{equation}
\label{f:005}
\begin{aligned}
(\Omega^{N}_{\mathrmtiny{ WFmut}}f)(\underline{x})
&= \sum_{1\le i\not =j\le N} \Big[f\big(\theta_{i,j}(\underline{x})\big)-f\big(\underline{x}\big)\Big]
+ \mu \sum_{i=1}^N \int^1_0\Big[f\big(\kappa^y_{i}(\underline{x})\big)-f\big(\underline{x}\big)\Big]\,\mathrm{d}y,
\end{aligned}
\end{equation}
where, for each $i,j \in [N]$ with $i\not =j$, $\theta_{i,j}$ denotes the {\em replacement operator} that sends the $N$-population $\underline{x}\in \mathbb{K}^N$ to the $N$-population $\theta_{i,j}(\underline{x})\in \mathbb{K}^N$ such that, for all $k \in [N]$,
\begin{equation}
\label{f:006}
\big(\theta_{i,j}(\underline{x})\big)_k := 
\begin{cases}
x_k, &\text{if } k\not=j, \\
x_i, &\text{if } k=j,
\end{cases}
\end{equation}
and, for $i \in [N]$ and $y\in \mathbb{K}$, $\kappa_i^y$ denotes the {\em mutation operator} that sends the $N$-population $\underline{x}\in \mathbb{K}^N$ to the $N$-population $\kappa^y_{i}(\underline{x})\in \mathbb{K}^N$ such that, for all $k \in [N]$,
\begin{equation}
\label{f:007}
\big(\kappa^y_{i}(\underline{x})\big)_k := 
\begin{cases}
x_k, &\text{if } k\not=i, \\
y, &\text{if } k=j.
\end{cases}
\end{equation}

We rely on the {\em resampling model with mutation} to provide a proof of Proposition~\ref{P:001}. For that, recall that ${\mathcal N}^N(\mathbb{N})$ denotes the set of integer-valued measures $\nu$ on $\mathbb{N}$ with $\sum_{k \in \mathbb{N}} k\nu(k)=N$, and consider the function $h\colon\,\mathbb{K}^N\to{\mathcal N}^N_f(\mathbb{N}_0)$ that evaluates for an $N$-population the frequency spectrum of types in the population, i.e., for $\underline{x}\in \mathbb{K}^N$,
\begin{equation}
\label{f:008}
h\big((x_1,\ldots,x_N)\big) := \sum_{y\in\{x_1,\ldots,x_N\}} \delta_{\#\{i\in[N]:\,x_i=y\}} \in {\mathcal N}(\mathbb{N}).
\end{equation}

\begin{lemma}[Frequency spectrum dynamics]
Let $Z:=(Z_t)_{t \ge 0}$ be given by $Z_t:=h(Y_t)$. Then $Z$ is an ${\mathcal N}^N(\mathbb{N})$-valued Markov chain whose generator $\Omega^N_{\mathrmtiny{ frequency}}$ acts on functions $F\colon\,{\mathcal N}^N(\mathbb{N})\to\mathbb{R}$ as follows:
\begin{equation}
\label{f:009}
\begin{aligned}
&(\Omega^N_{\mathrmtiny{ frequency}}F)(\nu)\\
&= \sum_{k_1=1}^{N-1}\sum_{k_2=2}^Nk_1\nu(k_1)k_2\big(\nu(k_2)-\delta_{k_1,k_2}\big)
\Big[F\big(\nu-\delta_{k_1}-\delta_{k_2}+\delta_{k_1+1}+\delta_{k_2-1}\big)-F\big(\nu\big)\Big]\\
&\quad + \sum_{k_1=1}^{N-1}k_1\nu(k_1)\big(\nu(1)-\delta_{k_1,1}\big)
\Big[F\big(\nu-\delta_{k_1}-\delta_{1}+\delta_{k_1+1}\big)-F\big(\nu\big)\Big]\\
&\quad + \mu\sum_{k=2}^Nk\nu(k)\Big[F\big(\nu-\delta_k+\delta_1+\delta_{k-1}\big)-F\big(\nu\big)\Big].
\end{aligned}
\end{equation}
\label{L:003}
\end{lemma}

\begin{proof}
For all $F\colon\,{\mathcal N}^N(\mathbb{N})\to\mathbb{R}$,
\begin{equation}
\begin{aligned}
\label{e:019}
&(\Omega^N_{\mathrmtiny{ WFmut}}F\circ h)\big((x_1,\ldots,x_N)\big)\\
&= \sum_{1\le i\not =j\le N} \Big[F\circ h\big(\theta_{i,j}(\underline{x})\big)-F\circ h\big(\underline{x}\big)\Big]
+ \mu \sum_{i=1}^N\int^1_0 \Big(=[F\circ h\big(\kappa^y_{i}(\underline{x})\big)-F\circ h\big(\underline{x}\big)\Big]\,\mathrm{d}y\\
&= \sum_{k_1=1}^{N-1}\sum_{k_2=2}^Nk_1\nu(k_1)k_2\big(\nu(k_2)-\delta_{k_1,k_2}\big)
\Big[F\big(h(\underline{x})-\delta_{k_1}-\delta_{k_2}+\delta_{k_1+1}+\delta_{k_2-1}\big)-F\big(h(\underline{x})\big)\Big]\\
&\quad + \sum_{k_1=1}^{N-1}k_1\nu(k_1)\big(\nu(1)-\delta_{k_1,1}\big)
\Big[F\big(h(\underline{x})-\delta_{k_1} -\delta_{1}+\delta_{k_1+1}\big)-F\big(h(\underline{x})\big)\Big]\\
&\quad + \mu\sum_{k_0=2}^Nk_0\nu(k_0) \Big[F\big(h(\underline{x})-\delta_{k_0}+\delta_1\big)-F\big(h(\underline{x})\big)\Big].
\end{aligned}
\end{equation}
As $h\colon\,\mathbb{K}^N\to{\mathcal N}^N(\mathbb{N})$ is surjective, the claim follows.
\end{proof}

Recall the MED-$N$-graph from Definition~\ref{Def:001}. The next lemma states that the MED distribution is invariant under the dynamics of the frequency spectrum under the resampling model with mutation. We believe this fact to be known, but were unable to find a reference in the literature and therefore provide a proof here.

\begin{lemma}[Frequency spectrum of the resampling model with mutation]
The multivariate Ewens distribution $\pi^N_{\mathrmtiny{ MED}}$ is the unique stationary distribution.
\label{L:001}
\end{lemma}

\begin{proof}
For $\tilde{\nu},\nu\in{\mathcal N}_N(\mathbb{N})$, once more denote by $q(\tilde{\nu},\nu)$ the rate to jump from $\tilde{\nu}$ to $\nu$, and put $\bar{q}(\nu) := \sum_{\nu'\in{\mathcal N}_N(\mathbb{N})} q(\nu,\nu')$ for the total rate to jump away from $\nu$. We start by showing that $\pi^N_{\mathrmtiny{ MED}}$ is indeed a stationary distribution, i.e., for all $\nu\in{\mathcal N}_N(\mathbb{N})$,
\begin{equation}
\label{f:120}
\sum_{\tilde{\nu}\in{\mathcal N}_N(\mathbb{N})}\pi^N_{\mathrmtiny{ MED}}\big(\tilde{\nu}\big)\,q(\tilde{\nu},\nu)
= \bar{q}(\nu)\pi^N_{\mathrmtiny{ MED}}(\nu).
\end{equation}
Note that
\begin{equation}
\label{f:014}
\begin{aligned}
&q(\tilde{\nu},\nu)\\
&= 
\begin{cases}
k_0\tilde{\nu}(k_0),
&\text{ if } \nu = \tilde{\nu}-\delta_{k_0}+\delta_1+\delta_{k_0-1}, \text{ for } k_0\in [N] \setminus \{ 1\} \\[1mm]
k_1\tilde{\nu}(k_1)k_2\big(\tilde{\nu}(k_2)-\delta_{k_1,k_2}\big),
&\text{ if } \nu=\tilde{\nu}-\delta_{k_2}-\delta_{k_1}+\delta_{k_1+1}+\delta_{k_2-1},\\
&\text{ for } k_1\in [N-1], k_2 \in [N] \setminus \{1\},\\[1mm]
k_1\tilde{\nu}(k_1)\big(\tilde{\nu}(1)-\delta_{k_1,1}\big),
&\text{ if } \nu=\tilde{\nu}-\delta_{1}-\delta_{k_1}+\delta_{k_1+1}, \text{ for }k_1\in \{N-1\},\\[1mm]
0,
&\text{ else;}
\end{cases}
\end{aligned}
\end{equation}
or, equivalently,
 \begin{equation}
\label{f:014b}
\begin{aligned}
&q(\tilde{\nu},\nu)\\
&= 
\begin{cases}
\mu k_0(\nu(k_0)+1),
&\text{ if } \nu=\tilde{\nu}-\delta_{k_0}+\delta_1+\delta_{k_0-1}, \text{ for } k_0\in [N] \setminus \{1\}\\[1mm]
k_1\big(\nu(k_1)+1+\delta_{k_1,k_2}-\delta_{k_1+1,k_2}\big)
&\text{ if } \nu=\tilde{\nu}-\delta_{k_2}-\delta_{k_1}+\delta_{k_1+1}+\delta_{k_2-1},\\
\hspace{1cm}\times k_2\big(\nu(k_2)+1-\delta_{k_1+1,k_2}\big),
&\text{ for } k_1\in [N-1],k_2\in [N] \setminus \{1\},\\[1mm]
k_1\big(\nu(k_1)+1+\delta_{k_1,1}\big)\big(\nu(1)+1\big),
&\text{ if } \nu=\tilde{\nu}-\delta_{1}-\delta_{k_1}+\delta_{k_1+1}, \text{ for }k_1\in [N-1],\\[1mm]
0,
&\text{ else.}
\end{cases}
\end{aligned}
\end{equation}
Moreover, for all $\nu\in{\mathcal N}_N(\mathbb{N})$,
\begin{equation}
\label{f:121}
\begin{aligned}
\bar{q}_\nu
&= \sum_{k_0=2}^N q_{\nu,\nu-\delta_{k_0}+\delta_1+\delta_{k_0-1}}
+\sum_{k_1=1}^{N-1} \sum_{k_2=2}^N q_{\nu,\nu-\delta_{k_1}-\delta_{k_2}+\delta_{k_1+1}+\delta_{k_2-1}}
+\sum_{k_1=1}^{N-1} q_{\nu,\nu-\delta_{k_1}-\delta_1+\delta_{k_1+1}}.
\end{aligned}
\end{equation}
Thus,
\begin{equation}
\label{f:121b}
\begin{aligned}
\bar{q}_\nu
&= \mu\big(N-\nu(1)\big)+\sum_{k_1=1}^{N-1}k_1\nu(k_1)\sum_{k_2=2,k_2\not=k_1}^Nk_2\nu(k_2)
+ \sum_{k_1=1}^{N-1} (k_1)^2\nu(k_1)\big(\nu(k_1)-1\big)\\
&\quad
+ \nu(1)N\,\mathbf{1}\{\nu(N)=0\}-\nu(1)\big(\nu(1)-1\big)\\
&= \mu\big(N-\nu(1)\big)+\sum_{k_1=1}^{N-1}k_1\nu(k_1)\big(N-\nu(1)-k_1\nu(k_1)+\nu(1)\delta_{k_1,1}\big)
+ \sum_{k_1=1}^{N-1}(k_1)^2\nu(k_1)\big(\nu(k_1)-1\big)\\
&\quad +\nu(1)N\mathbf{1}\{\nu(N)=0\}-\nu(1)\big(\nu(1)-1\big)\\
&= \mu\big(N-\nu(1)\big)+\big(N-\nu(1)\big)N\,\mathbf{1}\{\nu(N)=0\} + \nu(1)N\,\mathbf{1}\{\nu(N)=0\}-\nu(1)\big(\nu(1)-1\big)\\
&= \mu\big(N-\nu(1)\big)+N^2\,\mathbf{1}\{\nu(N)=0\}-\nu(1)\big(\nu(1)-1\big).
\end{aligned}
\end{equation}
We therefore find that
\begin{equation}
\label{f:015}
\begin{aligned}
&\sum_{\tilde{\nu}\in{\mathcal N}_N(\mathbb{N})}\pi^N_{\mathrmtiny{ MED}}\big(\tilde{\nu}\big)\,q\big(\tilde{\nu},\nu\big)\\
& = \mu \sum_{k_0=2}^N\pi^N_{\mathrmtiny{ MED}}\big(\nu+\delta_{k_0}-\delta_1-\delta_{k_0-1}\big)\,k_0\big(\nu(k_0)+1\big)\\
&\quad + \sum_{k_1=1}^{N-1}\sum_{k_2=2}^N\pi^N_{\mathrmtiny{ MED}}\big(\nu+\delta_{k_1}
+\delta_{k_2}-\delta_{k_1+1}-\delta_{k_2-1}\big)\,\\
&\qquad \times k_1\big(\nu(k_1)+1+\delta_{k_1,k_2}-\delta_{k_1+1,k_2}\big)k_2\big(\nu(k_2)+1-\delta_{k_1+1,k_2}\big)\big)\\
&\quad+\sum_{k_1=1}^{N-1}\pi^N_{\mathrmtiny{ MED}}\big(\nu+\delta_{k_1}+\delta_{1}-\delta_{k_1+1}\big)\,k_1\big(\nu(k_1)+1+\delta_{k_1,1}\big)\big(\nu(1)+1\big).
\end{aligned}
\end{equation}
For all $k_0\in [N] \setminus \{1\}$,
\begin{equation}
\label{f:013}
\begin{aligned}
\frac{\pi^N_{\mathrmtiny{ MED}}\big(\nu+\delta_{k_0}-\delta_1-\delta_{k_0-1}\big)}{\pi^N_{\mathrmtiny{ MED}}\big(\nu\big)}
&=\frac{1}{\mu}\nu(1)\big(\nu(k_0-1)-\delta_{2,k_0}\big)\frac{k_0-1}{k_0(\nu(k_0)+1)},
\end{aligned}
\end{equation}
while, for all $k_1\in [N-1]$ and $k_2\in [N] \setminus \{1\}$,
\begin{equation}
\label{f:017}
\begin{aligned}
&\frac{\pi^N_{\mathrmtiny{ MED}}\big(\nu+\delta_{k_1}+\delta_{k_2}-\delta_{k_1+1}
-\delta_{k_2-1}\big)}{\pi^N_{\mathrmtiny{ MED}}\big(\nu\big)}\\
&= \frac{(k_1+1-\delta_{k_1+1,k_2})(k_2-1+\delta_{k_1+1,k_2})\nu(k_1+1-\delta_{k_1+1,k_2})
\nu(k_2-1+\delta_{k_1+1,k_2})}{k_1k_2(\nu(k_1)+1+\delta_{k_1,k_2}-\delta_{k_1+1,k_2})(\nu(k_2)+1-\delta_{k_1+1,k_2})},
\end{aligned}
\end{equation}
and, for all $k_1\in [N-1]$,
\begin{equation}
\label{e:019*}
\begin{aligned}
\frac{\pi^N_{\mathrmtiny{ MED}}\big(\nu+\delta_{k_1}+\delta_{1}-\delta_{k_1+1}\big)}{\pi^N_{\mathrmtiny{ MED}}\big(\nu\big)}
&= \mu\frac{(k_1+1)\nu(k_1+1)}{k_1\big(\nu(k_1)+1+\delta_{k_1,1}\big)\big(\nu(1)+1\big)}.
\end{aligned}
\end{equation}
Thus, indeed,
\begin{equation}
\label{f:018}
\begin{aligned}
&\sum_{\tilde{\nu}\in{\mathcal N}_N(\mathbb{N})}\pi^N_{\mathrmtiny{ MED}}\big(\tilde{\nu}\big)\,q(\tilde{\nu},\nu)\\
&= \pi^N_{\mathrmtiny{ MED}}(\nu)\Big\{\nu(1)\sum_{k_0=2}^N(k_0-1)\big(\nu(k_0-1)-\delta_{2,k_0}\big)\\
&\quad + \sum_{k_1=1}^{N-1}\sum_{k_2=2}^N(k_1+1-\delta_{k_1+1,k_2})(k_2-1+\delta_{k_1+1,k_2})
\nu(k_1+1-\delta_{k_1+1,k_2})\nu(k_2-1+\delta_{k_1+1,k_2})\\
&\quad + \mu\sum_{k_1=1}^{N-1}(k_1+1)\nu(k_1+1)\Big\}\\
&= \pi^N_{\mathrmtiny{ MED}}(\nu)\Big\{\nu(1)N\,\mathbf{1}\{\nu(N)=0\}-\nu(1)\big(\nu(1)-1\big) +\mu\big(N-\nu(1)\big)\\
&\quad +\sum_{k_1=2}^{N} k_1(k_1-1)\nu(k_1)\nu(k_1-1)
+ \sum_{k_1=2}^{N} k_1\nu(k_1) \sum_{k_2=2,k_2\not=k_1}^N(k_2-1) \nu(k_2-1)\Big\}\\
&= \pi^N_{\mathrmtiny{ MED}}(\nu)\Big\{\nu(1)N\,\mathbf{1}\{\nu(N)=0\}-\nu(1)\big(\nu(1)-1\big) + \mu\big(N-\nu(1)\big)\\
&\quad + \sum_{k_1=2}^{N} k_1(k_1-1)\nu(k_1)\nu(k_1-1) +N\,\mathbf{1}\{\nu(N)=0\}\big(N-\nu(1)\big)\\
&\quad - \sum_{k_1=2}^{N} k_1(k_1-1)\nu(k_1)\nu(k_1-1)\Big\}\\
&= \pi^N_{\mathrmtiny{ MED}}(\nu)\bar{q}_{\nu}.
\end{aligned}
\end{equation}
Since $Z$ is irreducible, $\pi^N_{\mathrmtiny{ MED}}$ is the unique stationary distribution.
\end{proof}

The Moran model with resampling is linked to the poaching model with self-employment as follows.

\begin{proposition}[Link to $N$-graph dynamics]
Fix $N\in\mathbb{N}$. Given $G\in\mathbb{G}_{\mathrmtiny{ complete}}^N$ and $y=(y_1,\ldots,y_N)\in \mathbb{K}^N$ with $g(G)=h(y)$, the poaching model with self-employment $X$ starting in $X_0=G$ can be coupled with the resampling model with mutation $Y$ starting in $Y_0=y$ so that
\begin{equation}
\label{f:183}
g(G_t) = h(Y_t), \qquad t \geq 0.
\end{equation}
\label{P:011}
\end{proposition}

\begin{proof}
Let $\Pi^1=\{\Pi^1_{i,j}\colon\,1\le i\not =j\le N\}$ be a family of independent  rate-$1$ Poisson point processes on $[0,\infty)$ indexed by $i\not=j\in[N]$. Further, let $\Pi^2$ be a Poisson point process on $[N]\times[0,\infty)$ with intensity measure $n\otimes\mu\mathrm{d}t$, where $n$ denotes the counting measure on $\mathbb{N}$ and $\mathrm{d}t$ the Lebesgue measure on $[0,\infty)$. Think of $\Pi^1_{ij}$ as putting an atom at those times, at which individual $i$ replaces the type of individual $j$ by its own type, and of $\Pi^2$ as putting an atom at $(i,t)$ if a mutation event involving the individual $i$ happens at time $t$. Moreover, let $K=\{K_i(n)\colon \,n\in\mathbb{N},i\in[N]\}$ be of family of independent $[0,1]$-valued random variables uniformly distributed on $[0,1]$. Think of $K_i(n)$ as the type of individual $i$ between its $n^{\mathrm{th}}$ and $(n+1)^{\mathrm{st}}$ mutation. Let $(\Omega,{\mathcal A},\mathbb{P})$ be a probability space on which $\Pi^1$, $\Pi^2$ and $K$ are defined.

We next define the poaching model with self-employment and the Moran model with mutation on  $(\Omega,{\mathcal A},\mathbb{P})$ as solutions of an SED driven by $\Pi^1$, $\Pi^2$ and $K$. For that we think of a tuple $(i,t) \in [N]\times[0,\infty)$ as the individual (or vertex) $i$ at time $t$, and we denote by $\mathrm{pr}_{\mathrmtiny{ ind}}$ and $\mathrm{pr}_{\mathrmtiny{ time}}$ the maps that send $(i,j)$ to its individual name and time, respectively. As usual, we can use the above point processes to define, for all $0\le s\le t$ and $i,j\in[N]$, an ancestral relationship. For that, we declare $(i,s)$ to be an ancestor of $(j,t)$ if and only if there exist finitely many times $s=t_0 \le t_1 \le t_2 \le \cdots \le t_n \le t_{n+1}=t$ and finitely many $i=i_0,i_1,\ldots,i_n,i_{n+1}=j\in[N]$ such that $\Pi^{1}_{i_{l-1},i_{l}}(\{t_l\})=1$ for all $l \in [n+1]$ and $\Pi^2(\{i_l\}\times(t_l,t_{l+1}))=0$ for all $l \in \{0,\ldots,n\}$. We refer to $(i,s)$ as the earliest ancestor of $(j,t)$ in $[0,\infty)$ if $s\le s'$ for all other ancestors $(i',s')$ of $(j,t)$.

We start with the Moran model with mutation. Define, for $i\in\mathbb{N}$ and $t \ge 0$,
\begin{equation}
\label{f:184}
\tilde{Y}^N_i(t) := 
\begin{cases}
y_{A_{[0,t]}(i,t)}, &\text{if } \mathrm{pr}_{\mathrmtiny{ time}}(A_{[0,t]}(i,t))=0,\\[1mm]
K_{\mathrm{pr}_{\mathrmtiny{ ind}}(A_{[0,t]}(i,t))}(n),
&\text{if } \Pi^2(\{\mathrm{pr}_{\mathrmtiny{ ind}}(A_{[0,t]}(i,t))\} \times [0,\mathrm{pr}_{\mathrmtiny{ time}}(A_{[0,t]}(i,t))])=n.
\end{cases}
\end{equation}
In particular, $\tilde{Y}^N_i(0)=y_i$. Clearly, $\tilde{Y}^N=(\tilde{Y}^N_i(t))_{t \ge 0}$ is a Moran model with mutation.

Next, for $t \ge 0$ define the graph $\tilde{G}_t=(V,E_t)\in\mathbb{G}^N_{\text{complete}}$ by choosing an arbitrary isomorphism between $V$ and $[N]$ w.r.t.~$G=(V,E)$, i.e., a bijection $\gamma\colon\,V\to[N]$ such that $y_{\gamma^{-1}(v)}=y_{\gamma^{-1}(v')}$ whenever $\{v,v'\}\in E$. This is possible because $\# V=N$ and $g(G)=h(y)$ by assumption. For $t\ge 0$, let
\begin{equation}
\label{f:185}
E_t := \big\{\{v,v'\}\colon\,\,v,v'\in V,v\not=v',Y_{\gamma^{-1}(v)}(t)=Y_{\gamma^{-1}(v')}(t)\big\}.
\end{equation}
In particular, $E_0:=\{\{v,v'\}\colon\,\,v,v'\in V,v\not=v',y_{\gamma^{-1}(v)}=y_{\gamma^{-1}(v')}\}$. It follows immediately that
(\ref{f:183}) holds for all $t\ge 0$, and so it remains to show that $\tilde{G}=((V,E_t))_{\ge 0}$ is the poaching model with self-employment.

Note that the edge set process $(E_t)_{t\ge 0}$ jumps at the same times as the Moran model with mutation $(\tilde{Y}^N_t)_{t\ge 0}$. Moreover, as the latter is a pure-jump process, we can track the different jumps separately. If $\Pi^1_{i,j}(\{t\})=1$ for some $i\not=j\in[N]$, then $\tilde{Y}^N(t-)$ jumps to $\theta_{i,j}(\tilde{Y}^N(t-))$ (with $\theta_{i,j}$ as introduced in (\ref{f:006})). This jump happens at rate $1$ and translates into a jump from $(V,E_{t-})$ to $(V,E_{t-})^{(\gamma^{-1}(i),\gamma^{-1}(j))}$, i.e., the graph built from $(V,E_{t-})$  by $\gamma^{-1}(i)$ poaching $\gamma^{-1}(j)$ (compare with \eqref{f:001}). On the other hand, if $\Pi^2(\{(i,t)\})=1$ and $\Pi^2(\{i\}\times[0,t])=n$ for some $i\in[N]$ and $n\in\mathbb{N}$, then $\tilde{Y}^N(t-)$ jumps to $\kappa^{K_i(n)}_{i}(\tilde{Y}^N(t-))$ (with $\kappa^K_{i}(y)$ as introduced in (\ref{f:007})). This jump happens with rate $\mu$ and translates into a jump from $(V,E_{t-})$ to $(V,E_{t-})^{\gamma^{-1}(i)}$, i.e., the graph built from $(V,E_{t-})$  by letting $\gamma^{-1}(i)$ leave its component and start a singleton component  (compare with \eqref{f:001}). This implies that $(\tilde{G}_t)_{t \ge 0}$ is indeed the poaching model with self-employment. As $\tilde{G}_0=G\in\mathbb{G}^N_{\mathrmtiny{ complete}}$ and $\mathbb{G}^N_{\mathrmtiny{ complete}}$ is invariant under the dynamics, it follows also that (\ref{f:183}) holds.
\end{proof}

We close this section with a proof of Proposition~\ref{P:001} stated in the form of a Corollary~\ref{Cor:001}.

\begin{corollary}[Proof of Proposition~\ref{P:001}]
If $X$ starts in $G_0\in\mathbb{G}^N$ and $Y^N$ starts in $y\in\mathbb{K}^N$, then ${\mathcal L}(X_t)\Tto{\mathcal L}(G^N_{\mathrmtiny{ MED}})$.
\label{Cor:002}
\end{corollary}

\begin{proof}
Let $X$ start in $G\in\mathbb{G}^N_{\mathrmtiny{ complete}}$ and $\tilde{Y}^N$ start in $y\in\mathbb{K}^N$ with $g(G)=h(y)$. Then, since $\tilde{Y}^N$ is irreducible, for all $\nu\in{\mathcal N}^N(\mathbb{N})$,
\begin{equation}
\label{f:186}
\begin{aligned}
\mathbb{P}_{g(G)}\big(g(\tilde{G}(t))=\nu\big)=\mathbb{P}_{h(y)}\big(h(\tilde{Y}^N(t))=\nu\big)
\Tto\pi^N_{\mathrmtiny{ MED}}\big(\{\nu\}\big).
\end{aligned}
\end{equation}
As the poaching model with self-employment enters $\mathbb{G}^N_{\mathrmtiny{ complete}}$ in finite time and does not leave $\mathbb{G}^N_{\mathrmtiny{ complete}}$ afterwards, (\ref{f:186}) extends to initial $N$-graphs $G\in\mathbb{G}^N$, i.e.,
\begin{equation}
\label{f:186b}
\begin{aligned}
\mathbb{P}_{g(G)}\big(g(\tilde{G}(t))=\nu\big)\Tto\pi^N_{\mathrmtiny{ MED}}\big(\{\nu\}\big).
\end{aligned}
\end{equation}
\end{proof}


\section{Diffusions for graphons and graphemes}
\label{S:graphemes}

In this section, we provide preliminaries on graphons that arise as limits of dense finite graphs as the number of vertices tends to infinity (Subsection~\ref{Sub:graphons}). These graphons often show features of random graphs. For the purpose of the present paper, we are only interested in the subspace of those graphons that are non-random, and that will be referred to as graphemes (Subsection~\ref{Sub:graphemes}).


\subsection{The space of graphons}
\label{Sub:graphons}

Recall that an $N$-graph $G=(V,E)$ is an undirected graph with $\# V=N$. A class of functionals that allows us to evaluate $N$-graphs is the following class of {\em subgraph densities} (compare also with (\ref{f:122})).

\begin{definition}[Subgraph density]
{\rm For $N\in\mathbb{N}$, let $G=(V,E)$ be an $N$-graph. Consider a finite graph $F=(V_F,E_F)$. The \textit{$F$-subgraph density $\Phi^{F}$} evaluates the relative frequency of $F$ as a subgraph of $G$, i.e.,
\begin{equation}
\label{f:128}
\begin{aligned}
&\Phi^F(G) := \frac{\mathbf{1}\{\mathrm{inj}(V_F,V_G)\not=\emptyset\}}{\#\mathrm{inj}(V_F,V_G)}\\
&\times \sum_{\phi\in\mathrm{inf}(V_F,V_G)}\prod_{e=(v,v')\in E_F}\mathbf{1}_{E_G}\big(\{\phi(v),\phi(v')\}\big)
\prod_{e=(v,v')\in\complement E_F}\,\mathbf{1}_{\complement E_G}\big(\{\phi(v),\phi(v')\}\big),
\end{aligned}
\end{equation}
where $\mathrm{inj}(V,V')$ denotes the set of all injective maps from $V$ to $V'$ and $\complement E$ denotes the set of all two-point sets that are not connected by an edge.}\hfill$\spadesuit$
\label{Def:002}
\end{definition}

\begin{remark}[Subgraph counts]
{\rm Dense graph limit theory usually relies on subgraph counts that are defined via an injective homomorphism and lead to test functions of the form: for $N\in\mathbb{N}$ and $F\in\mathbb{G}^k$ with $k\ge N$,
\begin{equation}
\label{f:128b}
\begin{aligned}
\Psi^F(G) &:= \frac{\mathbf{1}\{\mathrm{inj}(V_F,V_G)\not=\emptyset\}}{\#\mathrm{inj}(V_F,V_G)}
\sum_{\gamma\in\mathrm{inf}(V_F,V_G)}\prod_{e=(v,v')\in E_F}\mathbf{1}_{E_G}\big(\{\gamma(v),\gamma(v')\}\big).
\end{aligned}
\end{equation}
Note that the class of test functions of the form (\ref{f:128}) and (\ref{f:128b}) generate the same algebra, i.e., each of them can be expressed as a linear combination of finitely many of the others.}\hfill$\spadesuit$
\label{Rem:002}
\end{remark}

Recall further from (\ref{f:187}) that we consider two finite graphs $G=(V,E)$ and $G'=(V',E')$ as equivalent if and only if  there is an isomorphism $\gamma\colon\,V\to V'$ with $\{v,w\}\in E$ if and only if  $\{\gamma(v),\gamma(w)\}\in E'$.

\begin{lemma}[On isomorphisms and subgraph densities]
Two finite graphs $G=(V,E)$ and $G'=(V',E')$ are equivalent iff  $\Phi^F(G)=\Phi^F(G')$ for all finite graphs $F$.
\label{L:004}
\end{lemma}

\begin{proof}
{\em ``$\Rightarrow$''}
Assume first that $G=(V,E)$ and $G'=(V',E')$ are equivalent. Then there exists an isomorphism $\gamma\colon\,V\to V'$ with $\{v,w\}\in E$ if and only if $\{\gamma(v),\gamma(w)\}\in E'$. Fix a finite graph $F$. Then, $\# V=\# V'$, and thus $\#\mathrm{inj}(V_F,V)=\#\mathrm{inj}(V_F,V')$. Therefore,
\begin{equation}
\label{f:191}
\begin{aligned}
&\Psi^F(G')\\
&= \frac{\mathbf{1}\{\mathrm{inj}(V_F,V')\not=\emptyset\}}{\#\mathrm{inj}(V_F,V')}\sum_{\phi'\in\mathrm{inj}(V_F,V')}
\prod_{e=(v,v')\in E_F} \mathbf{1}_{E'}\big(\{\phi'(v),\phi'(w)\}\big)\\
&= \frac{\mathbf{1}\{\mathrm{inj}(V_F,V)\not=\emptyset\}}{\#\mathrm{inj}(V_F,V)}\sum_{\phi'\in\mathrm{inf}(V_F,V')}
\prod_{e=(v,v')\in E_F}\mathbf{1}_{E}\big(\{\gamma^{-1}\circ\phi'(v),\gamma^{-1}\circ\phi'(w)\}\big)\\
&= \frac{\mathbf{1}\{\mathrm{inj}(V_F,V)\not=\emptyset\}}{\#\mathrm{inj}(V_F,V)}\sum_{\phi\in\mathrm{inf}(V_F,V)}\prod_{e=(v,v')\in E_F}\mathbf{1}_{E}\big(\{\phi(v),\phi(w)\}\big)\\
&= \Psi^F(G),
\end{aligned}
\end{equation}
where we use that every $\phi\in\mathrm{inj}(V_F,V)$ is of the form $\phi=\gamma^{-1}\circ\phi'$ for some unique $\phi'\in\mathrm{inj}(V_F,V')$. The same argument shows that $\Phi^F(G)=\Phi^F(G')$ for all finite graphs $F$.

{\em ``$\Leftarrow$'' }
Assume that $G=(V,E)$ and $G'=(V',E')$ are two finite graphs such that $\Phi^F(G)=\Phi^{F}(G')$ for all finite graphs $F$. Choose $F:=G$. Then $\mathrm{inj}(V,V)$ contains the identity and therefore does not equal the empty set. Moreover, all $\phi\in\mathrm{inj}(V,V)$ are isomorphisms between $G$ and $G$. As $\Phi^G(G)=\Phi^{G}(G')$, it follows that there exists an isomorphism between $V$ and $V'$, as claimed.
\end{proof}

Recall from (\ref{f:190}) the space $\mathbb{G}^{N,\simeq}$ of isomorphism classes, and put
\begin{equation}
\label{f:131}
\mathbb{G}^{\simeq}:=\bigcup_{N\in\mathbb{N}}\mathbb{G}^{N,\simeq}.
\end{equation}
Since subgraph densities are separating points in $\mathbb{G}^{\simeq}$, we can use them to induce a notion of convergence.

\begin{definition}[Convergence of subgraph densities]
{\rm Given a sequence $(G_n)_{n\in\mathbb{N}}$ in $\mathbb{G}^{\simeq}$, we say that \textit{$(G_n)_{n\in\mathbb{N}}$ converges to $G\in\mathbb{G}^{\simeq}$} if and only if $\Phi^F(G_n)\tno\Phi^F(G)$ for all $F\in\mathbb{G}^{\simeq}$.}\hfill$\spadesuit$
\label{Def:008}
\end{definition}

It is not hard to see that $\mathbb{G}^{\simeq}$ is not complete. For example, for $n\in\mathbb{N}$ consider the complete graph $G_n$ with $n$ vertices. Then all subgraph densities are constant and equal to $1$, and therefore converge, while the sequence $(G_n)_{n\in\mathbb{N}}$ cannot converge to an element in $\mathbb{G}^{\simeq}$ as their number of vertices tends to infinity. Nonetheless, the space $\mathbb{G}^{\simeq}$ can be completed by the following standard procedure: consider the family ${\mathcal V} := \{B_k, k\in\mathbb{N}\}$ of sets of the form
\begin{equation}
\label{f:132}
B_k := \big\{(G,G')\in\mathbb{G}^{\simeq}\times\mathbb{G}^{\simeq}\text{ with }\Phi^F(G)=\Phi^F(G'),\,\forall F\in \mathbb{G}^{k,\simeq}\big\}.
\end{equation}
Then, ${\mathcal V}$ is the base of a uniformity ${\mathcal U}_{\mathbb{G}^{\simeq}}$ on $\mathbb{G}^{\simeq}\times\mathbb{G}^{\simeq}$. As usual, if $(H_1,{\mathcal U}_1)$ and $(H_2,{\mathcal U}_2)$ are two uniform spaces, then we refer to a function $\Psi\colon\,H_1\to H_2$ as {\em uniformly continuous} iff $\Phi^{-1}(U_{2})\subseteq {\mathcal U}_1$ for all $U_{2}\in{\mathcal U}_2$, i.e., iff for all $U_2\in{\mathcal U}_2$ there exists $U_{1}\in{\mathcal U}_{1}$ such that $(G,G')\in U_{1}$ implies $(\Psi(G),\Psi(G'))\in U_2$. As with metric spaces, $\mathbb{G}^{\simeq}$ as a uniform space has a {\em Hausdorff completion}, i.e., there exists a complete Hausdorff uniform space $\widetilde{\mathbb{G}}^{\simeq}$ and a topological embedding ${\mathcal I}\colon\,\mathbb{G}^{\simeq}\to\widetilde{\mathbb{G}}^{\simeq}$ such that, for all uniformly continuous mappings $\Psi$ from $\mathbb{G}^{\simeq}$ into a uniform space $\mathbb{H}$, there exists a uniformly continuous map $\widetilde{\Psi}\colon \widetilde{\mathbb{G}}^{\simeq}\to\mathbb{H}$ with $\Psi=\widetilde{\Psi}\circ{\mathcal I}$.

Let $\widetilde{\mathbb{G}}^{\simeq}$ be the completion of $\mathbb{G}^{\simeq}$. In the theory of dense graph limits introduced by \cite{LovaszSzegedy2006} and further developed in \cite{BorgsChayesLovaszSosVesztergombi2008}, the equivalence classes in $\widetilde{\mathbb{G}}^{\simeq}$ are represented by {\em graphons} (see \cite[Theorem~9.1]{Janson2013}).

\begin{definition}[Graphons]
{\rm A \textit{graphon} is a quadruple $(\Omega,\tau,\mu,W)$ consisting of a
\begin{itemize}
\item Polish space $(\Omega,\tau)$,
\item probability measure $\mu$ on ${\mathcal B}(\Omega)$,
\item symmetric and Borel measurable function $W\colon\,\Omega^2\to[0,1]$.
\end{itemize}
Two graphons $(\Omega,\tau,\mu,W)$ and $(\Omega,\tau,\mu,W')$ are identified when $W=W'$ $\mu$-almost surely.  We refer to ${\mathcal W}(\Omega,\tau,\mu)$ as the set of graphons on $(\Omega,\tau,\mu)$.}\hfill$\spadesuit$
\label{f:133}
\end{definition}

\begin{remark}[Graphons as square-map]
{\rm Often the probability space $(\Omega,\tau,\mu)$ is taken to be $[0,1]$ together with the Euclidian topology and the Lebesgue measure $\lambda$ on $[0,1]$. In fact, as $W$ is a symmetric and Borel measurable function $W\colon\,\Omega^2\to[0,1]$, there exists a symmetric Borel measurable function $\tilde{W}\colon\,[0,1]^2\to[0,1]$ such that, for all continuous bounded functions $f\colon\,[0,1]\to\mathbb{R}$,
\begin{equation}
\label{f:137}
\int_{\Omega^2}f\big(W(\omega_1,\omega_2)\big)\,\mu^{\otimes 2}(\mathrm{d}(\omega_1,\omega_2))
= \int_{[0,1]^2}f\big(\tilde{W}(x_1,x_2)\big)\,\mathrm{d}(x_1,x_2)
\end{equation}
(\cite[Theorem~2]{DiaconisJanson2008}). Most papers on graph limits only consider this case. However, for our purposes it will be convenient to allow for a more general Borel probability space.}\hfill$\spadesuit$
\label{Rem:001}
\end{remark}

\begin{example}[Finite graphs]
Every $N$-graph $G=(V,E)$ can be represented by a graphon as follows. Let $\Omega=V$ be equipped with the discrete topology $\tau$,
$\mu:=\frac{1}{N}\sum_{v\in V}\delta_v$, and $W(v,v')=\mathbf{1}_E(\{v,v'\})$. To represent the $N$-graph $G$ via the Euclidian interval $[0,1]$ equipped with the Lebesgue measure, partition $(0,1]$ into $N$ intervals $(\frac{i-1}{N},\frac{i}{N}]$, $i \in [N]$. Choose a bijection $\phi\colon\,V\to[N]$, and let
\begin{equation}
\label{f:138}
\tilde{W}(x,y):=\sum_{i,j \in [N], i\not =j}\mathbf{1}_E\big(\{\phi^{-1}(i),\phi^{-1}(j)\}\big)\,
\mathbf{1}_{(\frac{i-1}{N},\frac{i}{N}]\times(\frac{j-1}{N},\frac{j}{N}]}\big((x,y)\big)
.
\end{equation}
The two graphons $W$ and $\tilde{W}$ have the same distribution in the sense of (\ref{f:138}).
\hfill$\qed$
\label{Exp:001}
\end{example}

Given a graphon $W$ on $(\Omega,\tau,\mu)$, we define for $F=(V_F,E_F)\in\mathbb{G}^k$, $k\in\mathbb{N}$, with $V_F=[k]$, the {\em subgraphon density}
\begin{equation}
\label{f:134}
t^F(W):=\int_{\Omega^k} \prod_{(i,j):\,\{i,j\}\in E_F}W(x_i,x_j)
\prod_{(i,j):\,\{i,j\}\in \complement E_F}\big(1-W(x_i,x_j)\big)\,\mu^{\otimes k}(\mathrm{d}(x_1,\ldots,x_k)).
\end{equation}
As before, we consider two graphons $W,W'\in{\mathcal W}$ as equivalent if and only if $t^F(W)=t^F(W')$ for all $F\in\mathbb{G}^{\simeq}$, and define
\begin{equation}
\label{f:135}
{\mathcal W}^{\simeq}:=\text{ set of all equivalence classes of graphons}.
\end{equation}

The following statement was proved in \cite{LovaszSzegedy2006} using Szemer\'{e}di partitions and the martingale convergence theorem. A more probabilistic proof using results from \cite{Hoover1979} and \cite{Aldous1981} can be found in \cite{DiaconisJanson2008}.

\begin{proposition}[Space of graphons]
The space ${\mathcal W}^{\simeq}$ equipped with the notion of convergence of subgraphon densities equals $\widetilde{G}^{\simeq}$, i.e., for every $G\in\widetilde{\mathbb{G}}^{\simeq}$, there is a unique (up to equivalence) graphon $(\Omega,\tau,\mu,W)$ such that, for all $F\in\mathbb{G}^{\simeq}$,
\begin{equation}
\label{f:136}
\Phi^{F}(G)=t^F(W).
\end{equation}
\label{P:003}
\end{proposition}

We will make use of the fact that the space of graphons is a compact Polish space (\cite[Theorem~5.1]{LovaszSzegedy2007}). One prominent metric metrizing the topology given by the notion of convergence of subgraphon densities is the so-called {\em cut metric} (compare with \cite{BorgsChayesLovaszSosVesztergombi2008,BorgsChayesLovaszSosVesztergombi2012,DiaoGuillotKhareRajaratnam2015}).


\subsection{The space of graphemes}
 \label{Sub:graphemes}

In what follows, we are interested in the subspace of those graphons from which we can read off an adjacency matrix together with a sampling measure on a possibly infinite vertex set.

\begin{definition}[Grapheme]
{\rm A graphon $(\Omega,\tau,\mu,W)\in\mathbb{G}^{\simeq}$ is called a \textit{grapheme} (or random-free graphon) if
$W(\Omega\times\Omega)\subseteq\{0,1\}$.}\hfill$\spadesuit$
\label{Def:009}
\end{definition}

Put
\begin{equation}
\label{f:139}
\mathbb{G}^{\simeq,\{0,1\}}:=\big\{(\Omega,\tau,\mu,W)\in\mathbb{G}^{\simeq}\colon\,W(\Omega\times\Omega)\subseteq\{0,1\}\big\}.
\end{equation}
Note that each $(\Omega,\tau,\mu,W)$ can be identified with a measured graph $G=(V,E,\mu)$ with vertex set $V:=\Omega$, Borel probability measure $\mu$ and $E:=\{\{v,v'\}\colon\,v,v'\in V,W(v,v')=1\}$.

\begin{example}[Complete graph versus (pseudo) random graph]
{\rm Consider once more the sequence $(G_N)_{N\in\mathbb{N}}$ of complete graphs whose number of vertices $N$ goes to infinity. Obviously, $(G_N)_{N\in\mathbb{N}}$ converges to the grapheme represented by $([0,1],\mathrm{d}x,\mathbf{1}_{[0,1]\times[0,1]})$. Consider the graphon represented by
\begin{equation}
\label{f:194}
W_{\frac{1}{2}} := ([0,1],\mathrm{d}x,\frac{1}{2}\mathbf{1}_{[0,1]\times[0,1]}).
\end{equation}
For two independent families $\{X_n;\,n\in\mathbb{N}\}$ and $\{Y_{i,j};\,i\not =j\in\mathbb{N}\}$ of independent random variables, uniform on $[0,1]$, respectively, $\{0,1\}$-valued Bernoulli with success parameter $\frac{1}{2}$, let $V_k:=\{X_1,\ldots,X_k\}$, $\tau_k$ the discrete topology, $\mu_k:=\frac{1}{k}\sum_{i=1}^k\delta_{X_i}$, and $E_k:=\{\{X_i,X_j\}:\,1\le i\not=j\le k,Y_{i,j}=1\}$. Obviously, $\Phi^{F}([(V_N,E_N)])\tNo t^F([W_{\frac{1}{2}}])$ for all finite graphs $F$, almost surely. We can therefore find realisations of the above random variables such that all subgraph densities converge to $W_{\frac{1}{2}}$. However, as we will see later, $W_{\frac{1}{2}}$ is not a grapheme.}\hfill$\spadesuit$
\label{Exp:002}
\end{example}

Example~\ref{Exp:002} shows that the space $\mathbb{G}^{\simeq,\{0,1\}}$ of graphemes is not complete. Since later we want to use $\mathbb{G}^{\simeq,\{0,1\}}$ as the state space of the diffusion limit of the $N$-graph-valued Markov chain, we are in need of a description of the compact sets. We use the following strategy to generalise the construction of Example~\ref{Exp:002} as follows. Let $(\Omega,\tau,\mu,W)$ be a graphon. Define a random graph $G(N,W)$ with vertex set $V=[N]$ by first taking families  $\{X_n\colon\,n\in\mathbb{N}\}$ and $\{Y_{i,j}\colon\,i\not =j\in\mathbb{N}\}$ of independent random variables, uniform on $[0,1]$, respectively, $\{0,1\}$-valued independent with $\mathbb{P}(Y_{i,j}=1 \mid \{X_i,X_j\})=W(X_i,X_j)$, i.e., we first sample $X_1, X_2, \ldots$ at random, and then toss a biased coin for each possible edge. As usual, we define the entropy of the discrete random variable $G(N,W)$ as
\begin{equation}
\label{f:192}
\begin{aligned}
\mathrm{Ent}\big(G(N,W)\big)
&:= -\sum_{F\in\mathbb{G}^N}\mathbb{P}\big(G(N,W)=F\big)\log\big(\mathbb{P}\big(G(N,W)=F\big)\big)\\
&= -\sum_{F\in\mathbb{G}^N}t^F\big([(\Omega,\tau,\mu,W)]\big)\log\big(t^F\big([(\Omega,\tau,\mu,W)]\big)\big).
\end{aligned}
\end{equation}

The following proposition relies on the characterisation of compact sets in the space of (equivalence classes of) graphemes given in \cite[Theorem~10.16]{Janson2013}.

\begin{proposition}[Compact sets]
Let $\Gamma\mathfrak{g}\subseteq\mathbb{G}^{\simeq,\{0,1\}}$ be a family of graphemes. Then, $\Gamma$ is relatively compact in $\mathbb{G}^{\simeq,\{0,1\}}$ if and only for all $\delta>0$ there exists $k_0=k_0(\delta)\in\mathbb{N}_0$ such that, for all $k\ge k_0$,
\begin{equation}
\label{f:173}
\sup_{\mathfrak{g}\in\Gamma}\sum_{F\in\mathbb{G}^k}h\big(t^{F}(\mathfrak{g})\big)\le\delta k^2,
\end{equation}
where $h(x)=-x\log(x)$.
\label{P:010}
\end{proposition}

\begin{proof}
Fix  a sequence $(\mathfrak{g}_n)_{n\in\mathbb{N}}$ in $\Gamma$. As $\widetilde{G}^{\cong}$ is compact, there exists a subsequence $(n_k)_{k\in\mathbb{N}}$ and a limit graphon $\mathfrak{g}\in\widetilde{\mathbb{G}}^{\cong}$ such that $\mathfrak{g}_{n_k}\tko\mathfrak{g}$. Assume that for all $\delta>0$ there exists a $k_0\in\mathbb{N}$ such that (\ref{f:173}) holds for all $k\ge k_0$. Then,
\begin{equation}
\label{f:172}
\lim_{k\to\infty}\frac{1}{{k\choose 2}}\sum_{F\in\mathbb{G}^k}h\big(t^F(\mathfrak{g})\big)=0.
\end{equation}
Thus, $\mathfrak{g}\in\widetilde{G}^{\cong,\{0,1\}}$ by \cite[Theorem~10.6]{Janson2013}.
\end{proof}

\begin{example}[Complete graph versus (pseudo) random graph sequence]
Consider the sequence $(G_N)_{N\in\mathbb{N}}$ of complete graphs whose number of vertices $N$ goes to infinity. Obviously, $(G_N)_{N\in\mathbb{N}}$ converges to the grapheme represented by $([0,1],\mathrm{d}x,\mathbf{1}_{[0,1]\times[0,1]})$. We also see that, for all $k\in\mathbb{N}$ and all finite graphs $F\in\mathbb{G}^k$, $t^F(\mathfrak{g}_N)\in\{0,1\}$, and so (\ref{f:173}) obviously holds. However, if for each $N\in\mathbb{N}$ the graph $G_N$ is such that $G_N\tNo W_{\frac{1}{2}}$ (with $W_{\frac{1}{2}}$ as defined in (\ref{f:194})), then, for all $k\in\mathbb{N}$,
\begin{equation}
\label{f:193}
\begin{aligned}
\sum_{F\in\mathbb{G}^k}h\big(t^F(\mathfrak{g})\big)
&= -\sum_{F\in\mathbb{G}^k}2^{-{k\choose 2}}\log\big(2^{-{k\choose 2}}\big)
= {k\choose 2}\log 2.
\end{aligned}
\end{equation}
Therefore, $W_{\frac{1}{2}}\not\in\mathbb{G}^{\simeq,\{0,1\}}$.
\hfill$\qed$
\label{Exp:002b}
\end{example}

In the poaching model with self-employment, we will make use of the following property of the space $\mathbb{G}_{\mathrmtiny{ complete}}$ of finite graphs consisting of only complete connected components.

\begin{proposition}[$\mathbb{G}_{\mathrmtiny{ complete}}$]
The space $\mathbb{G}_{\mathrmtiny{ complete}}$ is relatively compact in $\widetilde{\mathbb{G}}^{\simeq\{0,1\}}$.
\label{P:013}
\end{proposition}

\begin{proof}
Let $(\mathfrak{g}_N=[(\Omega_N,\tau_N,\mu_N,W_N)])_{N\in\mathbb{N}}$ be a sequence of graphons such that $\mathfrak{g}_N\tNo\mathfrak{g}$ for some $\mathfrak{g}\in\widetilde{\mathbb{G}}^{\simeq}$. We must show that $\mathfrak{g}=[(\Omega_N,\tau_N,\mu_N,W_N)]\in\widetilde{\mathbb{G}}^{\simeq\{0,1\}}$ for some $(\Omega_N,\tau_N,\mu_N,W_N)$ $\in\mathbb{G}_{\mathrmtiny{ complete}}$.

Fix $k\in\mathbb{N}$ and recall from (\ref{f:002}) the map $g\colon\,\mathbb{G}^k_{\mathrmtiny{ complete}}\to{\mathcal N}^k(\mathbb{N})$ with ${\mathcal N}^k(\mathbb{N})$ the set of all integer-valued measures with $\sum_{l=1}^kl\nu(\{l\})=k$. W.l.o.g.\ we may assume that, for each $N\in\mathbb{N}$, this is represented by an element in $\widetilde{\mathbb{G}}^{N,\simeq\{0,1\}}$. For each $k\in\mathbb{N}$, we find that
$G(k,W_N)\in \mathbb{G}^k_{\mathrmtiny{ complete}}$ and therefore
\begin{equation}
\label{f:196}
\begin{aligned}
&\mathrm{Ent}\big(G(k,W_N)\big)\\
&= -\sum_{F\in\mathbb{G}^k_{\mathrmtiny{ complete}}}\mathbb{P}\big(G(k,W_N)=F\big)\log\big(\mathbb{P}\big(G(k,W_N)=F\big)\big)\\
&= -\sum_{\nu\in{\mathcal N}^N(\mathbb{N})}\sum_{F\in\mathbb{G}^k_{\mathrmtiny{ complete}}\colon\,
g(F)=\nu}\mathbb{P}\big(G(k,W_N)=F\big)\log\big(\mathbb{P}\big(G(k,W_N)=F\big)\big)\\
&= -\sum_{\nu\in{\mathcal N}^N(\mathbb{N})}
\mathbb{P}\big(g(G(k,W_N))=\nu\big)\log\Big(\frac{\mathbb{P}\big(g(G(k,W_N))=\nu\big)}
{\#\big\{F\in\mathbb{G}^k_{\mathrmtiny{ complete}}:\,g(F)=\nu\big\}}\Big),
\end{aligned}
\end{equation}
where in the last line we use that  $\mathbb{P}(G(k,W_N)=F_1)=\mathbb{P}(G(k,W_N)=F_2)$ whenever $F_1,F_2\in \mathbb{G}^k$ with $g(F_1)=g(F_2)$. Note further that, for each $\nu\in{\mathcal N}^N(\mathbb{N})$,
\begin{equation}
\label{f:200}
\log\big(\#\big\{F\in\mathbb{G}^k_{\mathrmtiny{ complete}}:\,g(F)=\nu\big\}\big)
= \log\Big(\frac{k!}{\prod_{l=1}^k(l!)^{\nu(l)}(\nu(l))!}\Big)\le \log (k!) \le k\log k,
\end{equation}
and so
\begin{equation}
\label{f:199}
\sum_{\nu\in{\mathcal N}^N(\mathbb{N})}
\mathbb{P}\big(g(G(k,W_N))=\nu\big)\log\big(\#\big\{F\in\mathbb{G}^k:\,g(F)=\nu\big\}\big)\le k\log k.
\end{equation}
Hence
\begin{equation}
\label{f:201}
\begin{aligned}
\mathrm{Ent}\big(G(k,W_N)\big)
&\le \mathrm{Ent}\big(g\big(G(k,W_N)\big)\big)+k\log k\\
&\le \mathrm{Ent}\big(\text{Unif}(\#{\mathcal N}^k(\mathbb{N}))\big)+k\log k\\
&= \log\big(\#{\mathcal N}^k(\mathbb{N})\big)+k\log k.
\end{aligned}
\end{equation}

As ${\mathcal N}^k(\mathbb{N})$ has the same cardinality as the number of partitions of $k$ into a sum of positive integers, we have, as $k\to\infty$,
\begin{equation}
\label{f:202}
\begin{aligned}
\#{\mathcal N}^k(\mathbb{N})\sim\frac{\exp(\pi\sqrt{2k/3})}{4k\sqrt{3}}
\end{aligned}
\end{equation}
(see e.g.\ \cite{Baez1997}). This implies that
\begin{equation}
\label{f:203}
\begin{aligned}
\mathrm{Ent}\big(G(k,W_N)\big)
&\le \pi\sqrt{2k/3}+k\log k,
\end{aligned}
\end{equation}
which clearly satisfies (\ref{f:173}) and therefore settles the claim.
\end{proof}

We can say even more about limits of finite graphs with complete graph components. To that end, we call a symmetric measurable function $W\colon [0,1]^2\to[0,1]$ a \textit{diagonal block function} if there are (at most) countably many \textit{block sizes} $( \alpha_i )_{i \geq 1}$ with $1\ge\alpha_1\ge\alpha_2\ge \cdots \ge 0$, $\sum_{i \ge 1}\alpha_i\le 1$, $\sum_{n\ge 0}\alpha_n\le 1$ such that
\begin{equation}
\label{f:210}
W(x,y)=\sum_{n\ge 0}\mathbf{1}_{[\sum_{i=1}^n\alpha_i,\sum_{i=1}^{n+1}\alpha_{i})^2}\big(x,y\big),\hspace{1cm}x,y\in[0,1]
.
\end{equation}
We denote by
\begin{equation}
\label{f:211}
\widetilde{\mathbb{G}}^{\simeq, \{0,1\}}_{\mathrmtiny{ complete}}
:=\big\{\mathfrak{g}=[([0,1],\mu,W)]:\,W\text{ is a diagonal block function}\big\}
\end{equation}
the subspace of all graphemes that can be represented by a graph with complete graph components.

\begin{proposition}[A closed subspace] The space $\widetilde{\mathbb{G}}^{\simeq, \{0,1\}}_{\mathrmtiny{ complete}}$ is a closed subspace of the space $\widetilde{G}^{\simeq}$ of all graphons. In particular,  $\widetilde{\mathbb{G}}^{\simeq, \{0,1\}}_{\mathrmtiny{ complete}}$  is compact.
\label{P:015}
\end{proposition}

\begin{proof} Assume that $(\mathfrak{g}_N)_{N\in\mathbb{N}}$ is a sequence in $\widetilde{G}^{\simeq}$ that converges in terms of subgraphon converge to a graphon $\mathfrak{g}$. We want to show that $\mathfrak{g}\in\widetilde{G}^{\simeq}$. By approximation using finite block graphons, w.l.o.g., we can assume that for each $N\in\mathbb{N}$, 
$\mathfrak{g}_N=[([0,1],\mu_N,W_N)]$ is a finite diagonal block graphon with 
\begin{align}
\mu_{N} := \frac{1}{N}\sum_{i = 1}^{N} \delta_{\frac{i}{N}},
\end{align}
and 
\begin{align}
W_N(x,y)=\sum_{n = 1}^{M_N} \mathbf{1}_{[\sum_{i=1}^n \alpha^N_i,\sum_{i=1}^{n+1}\alpha^N_{i})^2}\big(x,y\big),\hspace{1cm}x,y\in[0,1]
,
\end{align}
where $\alpha^N_i = \frac{k_i}{N}$, $k_i \in [N]$, $i \in [M_N]$ and $\sum_{i=1}^{M_N} \alpha_i^N=1$.

Evidently, the convergence of block sizes $(\alpha^N_1,\alpha_2^N,\ldots)\tNo(\alpha_1,\alpha_2,\ldots)$ for some $\alpha_1\ge\alpha_2\ge \ldots \ge 0$ with $\sum_{i\ge 1}\alpha_i\le 1$ is equivalent to $\mathfrak{g}_N\tNo\mathfrak{g}$, where $\mathfrak{g}=[([0,1],\mathrm{d}x,W)]$ and $W$ is a diagonal block graphon with block sizes $(\alpha_1,\alpha_2,\ldots)$, as $N\to\infty$.
\end{proof}

Recall the MED-$N$-graph from Definition~\ref{Def:001}. It is known that as $N\to\infty$ the multivariate Ewens distribution converges in distribution to the {Griffith-Engen-McCloskey distribution (GEM)} (\cite{Engen1975}, \cite{McCloskey1965}, \cite{Griffiths1988}). To introduce GEM, let $V_1, V_2, \ldots$ be i.i.d.\ random variables that are beta-distributed with parameters $1$ and $\mu$, i.e., continuous $[0,1]$-valued random variables with density
\begin{equation}
\label{f:167}
f_Y(t)=\mu(1-t)^{\mu-1}.
\end{equation}
Put, for $k\in\mathbb{N}$,
\begin{equation}
\label{f:168}
Y_k:=V_k\prod_{i=1}^{k-1}(1-V_i).
\end{equation}
The random probability vector $(Y_1,Y_2,\ldots)$ is referred to as Griffith-Engen-McCloskey distribution (GEM).

\begin{definition}[GEM-grapheme]
{\rm The \textit{GEM-grapheme} is the random grapheme that can be represented by $\Omega:=[0,1]$ equipped with the Euclidian topology, the Lebesgue measure $\lambda$ on $\Omega$, and a random measurable symmetric map $W\colon\,\Omega^2\to\{0,1\}$ such that the following holds:
\begin{itemize}
\item
{\bf Complete components.}
For all $x,y,z\in[0,1]$, if $W(x,y)=1$ and $W(y,z)=1$, then $W(x,z)=1$. An equivalence relation $\sim_W$ is therefore given by $x\sim_W y$ if and only if $W(x,y)=1$. Moreover, the set $\mathrm{comp}_W$ consists of countably many measurable equivalence classes.
\item
{\bf GEM-component sizes.}
The random point process $\sum_{\varpi\in\mathrm{comp}_W}\delta_{\lambda(\varpi)}$ is equal in distribution to the random point process  $\sum_{k\in\mathbb{N}}\delta_{Y_k}$.
\end{itemize}
}\hfill$\spadesuit$
\label{Def:010}
\end{definition}

We close this subsection with stating that the MED-$N$-graph converges weakly to the GEM-grapheme.

\begin{proposition}[Convergence of the MED $N$-graph to the GEM-grapheme]
Let $\mathfrak{g}^N_{\mathrmtiny{ MED}}$ be the distribution of the MED-$N$-graph encoded as a grapheme, $N\in\mathbb{N}$, and $\mathfrak{g}_{\mathrmtiny{ GEM}}$ the distribution of the GEM-grapheme. Then,
\begin{equation}
\label{f:204}
{\mathcal L}\big(\mathfrak{g}^N_{\mathrmtiny{ MED}}\big)\TNo{\mathcal L}\big(\mathfrak{g}_{\mathrmtiny{ GEM}}\big).
\end{equation}
\label{P:014}
\end{proposition}

\begin{proof}
Since the subgraphon densities are convergence determining, we need to show that, for all $F\in\mathbb{G}^k$, $k\in\mathbb{N}$,
\begin{equation}
\label{f:205}
\mathbb{E}\Big[t^F\big(\mathfrak{g}^N_{\mathrmtiny{ MED}}\big)\Big]\TNo\mathbb{E}\Big[t^F\big(\mathfrak{g}_{\mathrmtiny{ GEM}}\big)\Big].
\end{equation}
To that end, fix $k\in\mathbb{N}$ and $F\in\mathbb{G}^k$.  As $t^F(\mathfrak{g}^N_{\mathrmtiny{ MED}})=0$ almost surely for all $N\in\mathbb{N}$ and $t^F(\mathfrak{g}_{\mathrmtiny{ GEM}})=0$ almost surely for all $F\in\mathbb{G}^k\setminus\mathbb{G}_{\mathrmtiny{ complete}}$, we may assume w.l.o.g.\ that $F\in\mathbb{G}^k_{\mathrmtiny{ complete}}$. Similarly to the proof of Proposition~\ref{P:013} (see in particular, Definition~\ref{f:010a} and (\ref{f:200})), we find that, for all $N\in\mathbb{N}$,
\begin{equation}
\label{f:206}
\begin{aligned}
\mathbb{E}\Big[t^F\big(\mathfrak{g}^N_{\mathrmtiny{ MED}}\big)\Big]
&= \mathbb{P}\big(G(k,\mathfrak{g}^N_{\mathrmtiny{ MED}})=F\big)\\
&= \frac{\mathbb{P}\big(g(G(k,\mathfrak{g}^N_{\mathrmtiny{ MED}}))=g(F)\big)}
{\#\{F'\in \mathbb{G}^k_{\mathrmtiny{ complete}}\colon\,g(F')=g(F)\}}\\
&= \pi^k_{\mathrmtiny{ graph}}\big(F\big),
\end{aligned}
\end{equation}
i.e., the MED-$N$-graph distribution is sampling consistent in the sense that sampling $k$ vertices of the random $N$-graph yields a random $k$-graph with the MED-$k$-graph distribution. It is also known that by sampling $k$ vertices $X_1,\ldots,X_k$ uniformly from $[0,1]$, and building an $k$-graph $\tilde{G}_{\mathrm{GEM}}$ by letting $i\sim j$ if and only if $Y_{l-1}\le X_i,X_j\le Y_l$ for some $l\in\mathbb{N}$, we get $\tilde{G}_{\mathrm{GEM}}\in\mathbb{G}^k_{\mathrmtiny{ complete}}$ satisfies
\begin{equation}
\label{f:208}
\begin{aligned}
\mathbb{E}\Big[t^F\big(\mathfrak{g}_{\mathrmtiny{ GEM}}\big)\Big]
&=\mathbb{P}\big(\tilde{G}_{\mathrmtiny{ GEM}}=F\big) = \pi^k_{\mathrmtiny{ MED}}(F)
\end{aligned}
\end{equation}
(compare~\cite[(41.4)]{JohnsonKotzBalakrishnan1997}). Hence (\ref{f:205}) follows.
\end{proof}


\section{The grapheme-valued Wright-Fisher diffusion with mutation}
\label{S:WFgrapheme}

In this section, we are interested in the limit dynamics of the poaching models with self-employment $X^N$ as the number of vertices $N$ goes to infinity. For that purpose, for each $N\in\mathbb{N}$, we encode $N$-graphs as graphemes with complete graph components (compare with Proposition~\ref{P:015}), and let
\begin{equation}
\label{f:152}
\widetilde{\mathbb{G}}_{\mathrmtiny{ complete}}^{N,\simeq,\{0,1\}}
:= \big\{\mathfrak{g}=[(\Omega,\tau,\mu,W)]\in\widetilde{\mathbb{G}}_{\mathrmtiny{ complete}}^{\simeq,\{0,1\}}\colon\,
\#\mathrm{supp}(\mu)=N\big\}
\end{equation}
be the space of {\em $N$-graphemes} with only complete components.

To identify a candidate for the grapheme-valued dynamics in the limit, recall from (\ref{f:004}) the operator  $\Omega^N_{\mathrmtiny{ adjacency}}$ acting on functions $\phi\colon\,\mathfrak{A}_k\to\mathbb{R}$ and from (\ref{f:125b}) how the operator $\Omega^N_{\mathrmtiny{ graphs}}$ acts on functions the subgraph densities $\Phi^{k,A}\colon\,\mathbb{G}\to\mathbb{R}$, $k\in\mathbb{N}$ and $A\in\mathfrak{A}_k$, i.e., for all $\mathfrak{g}=[(\Omega,\tau,\mu,W)]\in\widetilde{\mathbb{G}}^{N,\simeq,\{0,1\}}$,
\begin{equation}
\label{f:170}
\begin{aligned}
&(\Omega^N_{\mathrmtiny{ graph}}\Phi^{k,A})(\mathfrak{g})
= \int_{\Omega^k}\mu^{\otimes k,\downarrow}(\mathrm{d}(x_1,\ldots,x_k))\,
\Omega^k_{\mathrmtiny{ adjacency}}\mathbf{1}_A\big((W(x_i,x_j))_{1\le i,j\le k}\big),
 \end{aligned}
\end{equation}
where $\mu^{\otimes k,\downarrow}$ denotes independent sampling without repetition.

Consider the operator $\Omega_{\mathrmtiny{ grapheme}}$ acting on subgraph densities as follows: with $\mathfrak{g}=[(\Omega,\tau,\mu,W)]$ $\in\widetilde{\mathbb{G}}^{\simeq,\{0,1\}}$,
\begin{equation}
\label{f:144}
(\Omega_{\mathrmtiny{ grapheme}}\Phi^{k,A})(\mathfrak{g})
:= \int_{\Omega}\mu^{\otimes k}(\mathrm{d}(x_1,\ldots,x_k))\,(\Omega^k_{\mathrmtiny{ adjacency}}\,\mathbf{1}_A)\big((W(x_i,x_j))_{1\le i,j\le k}\big).
\end{equation}

\begin{proposition}[Convergence of the generators] For all $k\in\mathbb{N}$ and $A\in\mathfrak{A}_k$,
\begin{equation}
\label{f:171}
\lim_{N\to\infty}\sup_{\mathfrak{g}\in\widetilde{\mathbb{G}}^{N,\simeq,\{0,1\}}}
\big|(\Omega^N_{\mathrmtiny{ grapheme}}\Phi^{k,A})(\mathfrak{g})-(\Omega_{\mathrmtiny{ grapheme}}\Phi^{k,A})(\mathfrak{g})\big|=0.
\end{equation}
\label{P:008}
\end{proposition}

\begin{proof} This follows immediately, because for any given $k\in\mathbb{N}$, $(\Omega^k_{\mathrmtiny{ adjacency}}\mathbf{1}_A)(B)$ is bounded in $N$ uniformly over all $B\in\mathfrak{A}_k$, and the total variation distance of $\mu^{\otimes k,\downarrow}$ and $\mu^{\otimes k}$ converges to zero uniformly over sequences $(\mathfrak{g}_N)_{N\in\mathbb{N}}$ with $\mathfrak{g}_N\in\widetilde{\mathbb{G}}^{N,\simeq,\{0,1\}}$.
\end{proof}

\begin{theorem}[Wellposed martingale problem] Assume that $\mathfrak{g}\in\widetilde{\mathbb{G}}^{\simeq,\{0,1\}}_{\mathrmtiny{ complete}}$. Then, there exists a unique Markov process $V=(V_t)_{t\ge 0}$ with $V_0=\mathfrak{g}$ and continuous paths such that, for all $k\in\mathbb{N}$ and $A\in\mathfrak{A}_k$, the process $M^{k,A}:=(M^{k,A}_t)_{t\ge 0}$ given by
\begin{equation}
\label{f:145}
M^{k,A}_t=\Phi^{k,A}\big(V_t\big)-\Phi^{k,A}\big(\mathfrak{g}\big)-\int^t_0 (\Omega_{\mathrmtiny{ grapheme}}\Phi^{k,A})(V_s)\,\mathrm{d}s
\end{equation}
is a $\widetilde{\mathbb{G}}^{\simeq}_{\mathrmtiny{ complete}}$-valued martingale.
\label{T:001}
\end{theorem}

\begin{definition}[Grapheme-valued WF diffusion with mutation] We refer to the unique solution of the martingale problem from Theorem~\ref{T:001} as \textit{grapheme-valued Wright-Fisher diffusion with mutation}.
\label{f:146}\hfill$\spadesuit$
\end{definition}

The following proposition states that the family of $N$-graph dynamics are tight, provided we start in suitable initial conditions.

\begin{proposition}[Compact containment] Let $\{\widetilde{X}^N;\,N\in\mathbb{N}\}$ be the poaching model with self-employment with $\widetilde{X}^N(0)\tNo\widetilde{X}(0)$ for some $\widetilde{X}(0)\in\widetilde{\mathbb{G}}^{\simeq,\{0,1\}}$. Then, for all $T>0$ and $\varepsilon>0$, there exists a compact set $\Gamma:=\Gamma_{T,\varepsilon}\subseteq\mathbb{G}^{\simeq,\{0,1\}}$ such that
\begin{equation}
\label{f:195}
\begin{aligned}
\inf_{N\in\mathbb{N}}\mathbb{P}\big(\widetilde{X}^N(t)\in\Gamma,\,\forall t\in[0,T]\big)\ge 1-\varepsilon.
\end{aligned}
\end{equation}
\label{P:009}
\end{proposition}

\begin{proof}
As $\mathbb{G}^N_{\mathrmtiny{ complete}}$ is invariant under the poaching model with self-employment, the claim follows from Proposition~\ref{P:013}.
\end{proof}

To show uniqueness, we will rely on duality. Recall the Feynman-Kac duality from Proposition~\ref{P:006} the Feynman-Kac duality for the Markov chain with values in the space of adjacency matrices. Due to the algebraic from of the operator $\Omega_{\mathrmtiny{ grapheme}}$ we can lift this to our grapheme-valued process.

\begin{proposition}[Feynman-Kac duality]
\label{prop:duality}
 Let  $V=(V_t)_{t\ge 0}$
be a solution of the $\Omega_{\mathrmtiny{ grapheme}}$-martingale problem, and  $\overleftarrow{M}=(\overleftarrow{M}_t)_{t\ge 0}$
the $\mathfrak{A}_k$-valued Markov chain with duplication and grounding running in backward time. Then,
\begin{equation}
\label{f:148}
\begin{aligned}
&\mathbb{E}_{\mathfrak{g}}\Big[\int_{\Omega_t^k}\mu_t^{\otimes k}(\mathrm{d}\underline{x})\,
\mathbf{1}_A\big((W_t(x_i,x_j))_{1\le i\not =j\le k}\big)\Big]\\
&= \mathbb{E}_{A}\Big[\int_{\Omega^k}\mu^{\otimes k}(\mathrm{d}\underline{x})\,
\mathbf{1}_{(W(x_i,x_j))_{1\le i\not =j\le k}}\big(\overleftarrow{M}_t\big)
\exp\big(\int_0^t\beta^\mu_N(\overleftarrow{M}_s)\,\mathrm{d}s\big)\Big].
\end{aligned}
\end{equation}
\label{P:007}
\end{proposition}

\begin{proof}
Put
\begin{equation}
\label{f:151}
H\big(\mathfrak{g},(k,A)\big):=\Phi^{k,A}(\mathfrak{g}).
\end{equation}
Then, by (\ref{f:126}) for $k\in\mathbb{N}$, $A\in\mathfrak{A}_k$ and $\mathfrak{g}=(\Omega,\tau,\mu,W)\in\mathbb{G}^{\simeq}$,
\begin{equation}
\label{f:150}
\begin{aligned}
&(\Omega_{\mathrmtiny{ grapheme}}H)\big(\mathfrak{g},(k,A)\big)\\
&:= \int_{\Omega}\mu^{\otimes k}(\mathrm{d}\underline{x})\,\Omega^k_{\mathrmtiny{ adjacency}}\,
\mathbf{1}_A\big((W(x_i,x_j))_{1\le i,j\le k}\big)\\
&= \int_{\Omega}\mu^{\otimes k}(\mathrm{d}\underline{x})\,\Big(\overleftarrow{\Omega}^k_{\mathrmtiny{ adjacancy}}\,
\mathbf{1}_{(W(x_i,x_j))_{1\le i,j\le k}}\big(A\big)+\beta^\mu_k(A)\mathbf{1}_A\big((W(x_i,x_j))_{1\le i,j\le k}\big)\Big)\\
&= (\overleftarrow{\Omega}^k_{\mathrmtiny{ adjacancy}}H)\big(\mathfrak{g},(k,A)\big)+\beta^\mu_k(A)H\big(\mathfrak{g},(k,A)\big).
\end{aligned}
\end{equation}
As $\beta^\mu_k(A)\le k(k-1+\mu)<\infty$ for all $k\in\mathbb{N}$ and $A\in\mathfrak{A}_k$, the claim follows.
\end{proof}

The next proposition states that the limit of the stationary distribution of the poaching model with self-employment is the stationary distribution of the grapheme valued limit diffusion.
\begin{proposition}[Stationary distribution] Let $G=(G_t)_{t\ge 0}$ be the unique solution of the $\Omega_{\mathrmtiny{grapheme}}$-martingale problem. Then, the GEM-grapheme is the unique stationary distribution of $G$.
\label{P:009*}
\end{proposition}

\begin{proof} For each $k\in\mathbb{N}$ and $A\in\mathfrak{A}_k$, we have for $\mathfrak{g}=(\Omega,\tau,\mu,W)\in\mathbb{G}$
\begin{equation}
\label{f:169}
\begin{aligned}
&\mathbf{E}_{\mathrmtiny{ GEM}}\big[(\Omega_{\mathrmtiny{ grapheme}}\Phi^{F})\big(\mathfrak{g}\big)\big]\\
&= \mathbf{E}_{\mathrmtiny{ GEM}}\Big[\int_{\Omega^k}\mu^{\otimes k}(\mathrm{d}\underline{x})\,
(\Omega^k_{\mathrmtiny{ adjacency}}\mathbf{1}_A)\big(\big(W(x_i,x_j)\big)_{1\le i,j\le k}\big)\Big]\\
&= \int_{\Omega^k}\mu^{\otimes k}(\mathrm{d}\underline{x})\,\mathbf{E}_{\mathrmtiny{ MED}}\Big[(\Omega^k_{\mathrmtiny{ adjacency}}\,
\mathbf{1}_A)\big(\big(W(x_i,x_j)\big)_{1\le i,j\le k}\big)\Big] = 0,
\end{aligned}
\end{equation}
by the proof of Proposition~\ref{P:014}.
\end{proof}

\begin{proof}[Proof of Theorem~\ref{T:001}] \textit{Uniqueness} follows from duality, cf.,~Proposition~\ref{prop:duality}. 

{\em Tightness} follows, in view of Propositions~\ref{P:008} and~\ref{P:009}, with the exactly same proof as Theorems 3.9.1 and~3.9.4 in \cite{EthierKurtz1986} (see also \cite[Remark 4.5.2]{EthierKurtz1986}). In particular, it follows that if $\{\widetilde{X}^N;\,N\in\mathbb{N}\}$ is the poaching model with self-employment with $\widetilde{X}^N(0)\tNo\widetilde{X}(0)$ for some $\widetilde{X}(0)\in\widetilde{\mathbb{G}}^{\simeq,\{0,1\}}$, then $\widetilde{X}^N$ converges as $N\to\infty$ weakly on Skorokhod space to the grapheme-valued Wright-Fisher diffusion with mutation starting in $\widetilde{X}(0)$.

To argue for {\em continuous paths}, notice that the class of sample subgraphon densities is convergence determining. Hence, it is enough to show that $(\Phi^F(V_t))_{t\ge 0}$ has continuous paths for all $k\in\mathbb{N}$ and $F\in\mathbb{G}^k$. This follows from the convergence of the poaching model with self-employment to the grapheme-valued Wright-Fisher diffusion and the fact that $|\widetilde{X}^N(t-)-\widetilde{X}^N(t)|\le\frac{k}{N}$ (compare, e.g., \cite[Theorem~3.10.2]{EthierKurtz1986}).
\end{proof}

\bibliographystyle{alpha}
\bibliography{graphon}

\newcommand{\etalchar}[1]{$^{#1}$}
\begin{thebibliography}{GdHKW23}

\bibitem[AdHR21]{AthreyadenHollanderRoellin2021}
S.~Athreya, F.~den Hollander, and A.~Roellin.
\newblock Graphon-valued stochastic processes from population genetics, 2021.

\bibitem[Ald81]{Aldous1981}
D.J. Aldous.
\newblock Representations for partially exchangeable arrays of random
  variables.
\newblock {\em J. Multivariate Anal.}, 11:581--598, 1981.

\bibitem[Ant74]{Antoniak1974}
C.E. Antoniak.
\newblock Mixtures of {D}irichlet processes with applications to {B}ayesian
  non-parametric problems.
\newblock {\em Ann. Stat.}, 2:1152--1174, 1974.

\bibitem[BCL{\etalchar{+}}08]{BorgsChayesLovaszSosVesztergombi2008}
C.~Borgs, J.T. Chayes, L.~Lov\'{a}sz, V.T. S\'{o}s, and K.~Vesztergombi.
\newblock Convergent sequences of dense graphs {I}: Subgraph frequencies,
  metric properties and testing.
\newblock {\em Adv. Math}, 219:1801--1851, 2008.

\bibitem[BCL{\etalchar{+}}12]{BorgsChayesLovaszSosVesztergombi2012}
C.~Borgs, J.T. Chayes, L.~Lov\'{a}sz, V.T. S\'{o}s, and K.~Vesztergombi.
\newblock Convergent sequences of dense graphs {II}: Multiway cuts and
  statistical physics.
\newblock {\em Ann. Math.}, 176:151--219, 2012.

\bibitem[BD97]{Baez1997}
L.~B\'{a}ez-Duarte.
\newblock Hardy-{R}amanujan's asymptotic formula for partitions and the central
  limit theorem.
\newblock {\em Adv. Math.}, 125:114--120, 1997.

\bibitem[DGKR15]{DiaoGuillotKhareRajaratnam2015}
P.~Diao, D.~Guillot, A.~Khare, and B.~Rajaratnam.
\newblock Differential calculus on graphon space.
\newblock {\em J. Combin. Theory Ser. A}, 133:183--227, 2015.

\bibitem[DJ08]{DiaconisJanson2008}
P.~Diaconis and S.~Janson.
\newblock Graph limits and exchangeable random graphs.
\newblock {\em Rend. Mat. Appl.}, 28(7):33--61, 2008.

\bibitem[DT86]{DonnellyTavari1986}
P.~Donnelly and S.~Tavari.
\newblock The ages of alleles and a coalescent.
\newblock {\em Advances in Applied Journal}, 18:1--19, 1986.

\bibitem[EK86]{EthierKurtz1986}
S.N. Ethier and T.G. Kurtz.
\newblock {\em Markov processes}.
\newblock Wiley Series in Probability and Mathematical Statistics: Probability
  and Mathematical Statistics. John Wiley \& Sons, Inc., New York, 1986.
\newblock Characterization and convergence.

\bibitem[Eng75]{Engen1975}
S.~Engen.
\newblock A note on the geometric series as a species frequency model.
\newblock {\em Biometrika}, 62(3):697--699, 1975.

\bibitem[Ewe72]{Ewens1972}
W.J. Ewens.
\newblock The sampling theory of selectively neutral alleles.
\newblock {\em Theor. Popul. Biol.}, 3:87--112, 1972.

\bibitem[GdHKW23]{GrevendenHollanderKlimovskyWinter2023}
A.~Greven, F.~den Hollander, A.~Klimovsky, and A.~Winter.
\newblock Continuum graph dynamics via population dynamics: well-posedness,
  duality and equilibria.
\newblock In progress, 2023.

\bibitem[GPW09]{GrevenPfaffelhuberWinter2009}
A.~Greven, P.~Pfaffelhuber, and A.~Winter.
\newblock Convergence in distribution of random metric measure spaces
  ({$\Lambda$}-coalescent measure tree).
\newblock {\em Probab. Theory Relat. Fields}, 145(1):285--322, 2009.

\bibitem[Gri88]{Griffiths1988}
R.~Griffiths.
\newblock On the distribution of points in a {P}oisson-{D}irichlet process.
\newblock {\em J. Appl. Prob.}, 25:336--345, 1988.

\bibitem[Hoo79]{Hoover1979}
D.~Hoover.
\newblock Relations on probability spaces and arrays of random variables, 1979.
\newblock preprint.

\bibitem[Jan13]{Janson2013}
S.~Janson.
\newblock Graphons, cut norm and distance, couplings and rearrangements.
\newblock {\em NYJM Monographs}, 2013.

\bibitem[JKB97]{JohnsonKotzBalakrishnan1997}
N.L. Johnson, S.~Kotz, and N.~Balakrishnan.
\newblock {\em Discrete Multivariate Distributions}.
\newblock Wiley Series in Probability and Mathematical Statistics: Probability
  and Mathematical Statistics. John Wiley \& Sons, Inc., New York, 1997.

\bibitem[LMW20]{LoehrMytnikWinter2020}
W.~L\"ohr, L.~Mytnik, and A.~Winter.
\newblock Aldous move on cladograms in the diffusion limit.
\newblock {\em Annals Probab.}, 48(5):2565--2590, 2020.

\bibitem[LS06]{LovaszSzegedy2006}
L.~Lov\'{a}sz and B.~Szegedy.
\newblock Limits of dense graph sequences.
\newblock {\em J. Combin. Theory, Series B}, 96:933--957, 2006.

\bibitem[LS07]{LovaszSzegedy2007}
L.~Lov\'{a}sz and B.~Szegedy.
\newblock Szemeredi's lemma for the analyst.
\newblock {\em Geom. Funct. Anal.}, 17:252–--270, 2007.

\bibitem[LVW15]{LoehrVoisinWinter2015}
W.~L{\"o}hr, G.~Voisin, and A.~Winter.
\newblock Convergence of bi-measure {$\mathbb R$-trees} and the pruning
  process.
\newblock {\em Ann. Inst. H. Poincar\'e Probab. Statist.}, 51(4):1342--1368,
  2015.
\newblock arXiv:1304.6035.

\bibitem[LW21]{LoehrWinter2021}
W.~L\"{o}hr and A.~Winter.
\newblock Spaces of algebraic measure trees and triangulations of the circle.
\newblock {\em Bulletin FMS}, 149:55--117, 2021.

\bibitem[McC65]{McCloskey1965}
J.W. McCloskey.
\newblock A model for the distribution of individuals by species in an
  environment, 1965.
\newblock unpublished Ph.D. thesis, Michigan State University.

\bibitem[Sei15]{Seidel2015}
P.~Seidel.
\newblock The historical {M}oran model, 2015.
\newblock arXiv:1511.05781.

\end{thebibliography}

 \end{document}